\documentclass[11pt, letterpaper, twoside]{amsart}
   
\usepackage[%pagebackref,
hypertexnames=false, colorlinks, 
linkcolor=black,citecolor=black,urlcolor=black, linktocpage=true
]%
{hyperref} %,hypertexnames=false,colorlinks,[pagebackref]
\hypersetup{bookmarksdepth=3}

\usepackage{amssymb}
\usepackage{mathrsfs}
	\normalbaselines						  %

\renewcommand{\labelenumi}{(\roman{enumi})}

\newcommand{\shto}{\raisebox{.3ex}{$\scriptscriptstyle\rightarrow$}\!}

\newcommand{\R}{{\mathbb  R}}
\newcommand{\kf}[2]{{\begin{pmatrix} #1 \\ #2 \end{pmatrix}}}

\newcommand{\D}{{\mathbb  D}}

\newcommand{\T}{\mathbb{T}}

\newcommand{\Z}{{\mathbb  Z}}
\newcommand{\N}{{\mathbb  N}}
\newcommand{\C}{{\mathbb  C}}

\newcommand{\OZ}{{\mathbf{O}}}
\newcommand{\dd}{{d}}
\newcommand{\ID}{{\mathbf{1}}}

\newcommand{\I}{\mathbf{I}}

\newcommand{\fdot}{\,\cdot\,}

\newcommand{\wt}{\widetilde}

\newcommand{\cH}{\mathcal{H}}
\newcommand{\cD}{\mathcal{D}}
\newcommand{\cX}{\mathcal{X}}
\newcommand{\K}{\mathcal{K}}
\newcommand{\cM}{\mathcal{M}}
\newcommand{\cK}{\mathcal{K}}
\newcommand{\cU}{\mathcal{U}}
\newcommand{\cV}{\mathcal{V}}

\newcommand{\te}{\theta}
\newcommand{\f}{\varphi}

\DeclareMathOperator{\clos}{clos}

\DeclareMathOperator{\ran}{Ran}

\newcommand{\fD}{\mathfrak{D}}

\DeclareMathOperator{\rank}{rank}
\DeclareMathOperator{\Ker}{Ker}
\DeclareMathOperator{\Ran}{Ran}
\DeclareMathOperator{\spa}{\overline{span}}
\DeclareMathOperator{\im}{Im}
\DeclareMathOperator{\re}{Re}
\DeclareMathOperator{\supp}{supp}
\DeclareMathOperator{\dist}{dist}

\newcommand{\1}{\mathbf{1}}
\newcommand{\la}{\lambda}

\DeclareMathOperator{\ext}{Ext}

\newcommand{\ci}[1]{_{ {}_{\scriptstyle #1}}}
\newcommand{\ti}[1]{_{\scriptstyle \text{\rm #1}}}
\newcounter{vremennyj}

\newcommand\cond[1]{\setcounter{vremennyj}{\theenumi}\setcounter{enumi}{#1}\labelenumi\setcounter{enumi}{\thevremennyj}}

\catcode`@=11
\chardef\mathlig@atcode\count255

\def\actively#1#2{\begingroup\uccode`\~=`#2\relax\uppercase{\endgroup#1~}}
\def\mathlig@gobble{\afterassignment\mathlig@next@cmd\let\mathlig@next= }
\def\mathlig@delim{\mathlig@delim}
\def\mathlig@defcs#1{\expandafter\def\csname#1\endcsname}
\def\mathlig@let@cs#1#2{\expandafter\let\expandafter#1\csname#2\endcsname}
\def\mathlig@appendcs#1#2{\expandafter\edef\csname#1\endcsname{\csname#1\endcsname#2}}

\def\mathlig#1#2{\mathlig@checklig#1\mathlig@end\mathlig@defcs{mathlig@back@#1}{#2}\ignorespaces}

\def\mathlig@checklig#1#2\mathlig@end{%
 \expandafter\ifx\csname mathlig@forw@#1\endcsname\relax
 \expandafter\mathchardef\csname mathlig@back@#1\endcsname=\mathcode`#1%
 \mathcode`#1"8000\actively\def#1{\csname mathlig@look@#1\endcsname}%
 \mathlig@dolig#1\mathlig@delim
\fi
\mathlig@checksuffix#1#2\mathlig@end
}

\def\mathlig@checksuffix#1#2\mathlig@end{%
\ifx\mathlig@delim#2\mathlig@delim\relax\else\mathlig@checksuffix@{#1}#2\mathlig@end\fi
}
\def\mathlig@checksuffix@#1#2#3\mathlig@end{%
\expandafter\ifx\csname mathlig@forw@#1#2\endcsname\relax\mathlig@dosuffix{#1}{#2}\fi
\mathlig@checksuffix{#1#2}#3\mathlig@end
}

\def\mathlig@dosuffix#1#2{%
\mathlig@appendcs{mathlig@toks@#1}{#2}%
\mathlig@dolig{#1}{#2}\mathlig@delim
}

\def\mathlig@dolig#1#2\mathlig@delim{%
 \mathlig@defcs{mathlig@look@#1#2}{%
 \mathlig@let@cs\mathlig@next{mathlig@forw@#1#2}\futurelet\mathlig@next@tok\mathlig@next}%
 \mathlig@defcs{mathlig@forw@#1#2}{%
  \mathlig@let@cs\mathlig@next{mathlig@back@#1#2}%
  \mathlig@let@cs\checker{mathlig@chck@#1#2}%
  \mathlig@let@cs\mathligtoks{mathlig@toks@#1#2}%
  \expandafter\ifx\expandafter\mathlig@delim\mathligtoks\mathlig@delim\relax\else
  \expandafter\checker\mathligtoks\mathlig@delim\fi
  \mathlig@next
 }%
 \mathlig@defcs{mathlig@toks@#1#2}{}%
 \mathlig@defcs{mathlig@chck@#1#2}##1##2\mathlig@delim{%
  \ifx\mathlig@next@tok##1%
   \mathlig@let@cs\mathlig@next@cmd{mathlig@look@#1#2##1}\let\mathlig@next\mathlig@gobble
  \fi
  \ifx\mathlig@delim##2\mathlig@delim\relax\else
   \csname mathlig@chck@#1#2\endcsname##2\mathlig@delim
  \fi
 }%
 \ifx\mathlig@delim#2\mathlig@delim\else
  \mathlig@defcs{mathlig@back@#1#2}{\csname mathlig@back@#1\endcsname #2}%
 \fi
}%

\catcode`@\mathlig@atcode

\mathchardef\ordinarycolon\mathcode`\:
\def\vcentcolon{\mathrel{\mathop\ordinarycolon}}
\mathlig{:=}{\vcentcolon=}
\mathlig{::=}{\vcentcolon\vcentcolon=}
\numberwithin{equation}{section}

\theoremstyle{plain}
\newtheorem{theo}{Theorem}[section]
\newtheorem{cor}[theo]{Corollary}
\newtheorem{lem}[theo]{Lemma}
\newtheorem{prop}[theo]{Proposition}

\theoremstyle{definition}
\newtheorem{defn}[theo]{Definition}

\theoremstyle{remark}
\newtheorem*{ex*}{Example}
\theoremstyle{remark}
\newtheorem*{exs*}{Examples}
\theoremstyle{remark}
\newtheorem*{rem*}{Remark}
\newtheorem{rem}[theo]{Remark}
\theoremstyle{remark}
\newtheorem*{rems*}{Remarks}

\title[Clark model]{Clark model in general situation}
\author{Constanze~Liaw}
\thanks{CL is supported by the NSF grant DMS-1101477.}
\address{C.~Liaw: Department of Mathematics, Baylor University, One Bear Place \#97328,      
 Waco, TX  76798, USA}
\email{Constanze$\underline{\,\,\,}$Liaw@baylor.edu}
\author{Sergei~Treil}
\address{S.~Treil: Department of Mathematics, Brown University   
151 Thayer
Str./Box 1917,      
 Providence, RI  02912, USA}
 
\email{treil@math.brown.edu}

 \thanks{%This material is based on the work 
Work of S.~Treil is supported by the National Science Foundation under the grants  DMS-0800876, DMS-1301579. Any opinions, findings and conclusions or recommendations expressed in this material are those of the author and do not necessarily reflect the views of the National Science Foundation. }


\begin{document}

\begin{abstract}
For a unitary operator the family of its unitary perturbations by rank one operators with fixed range is parametrized by a complex parameter $\gamma$, $|\gamma|=1$. Namely all such unitary perturbations are the operators $U_\gamma:=U+(\gamma-1) ( \fdot,   b_1 )\ci{\cH}  b$, where $b\in \cH$, $\|b\|=1$, $b_1=U^{-1} b$, $|\gamma|=1$. For $|\gamma|<1$ the operators $U_\gamma$ are contractions with one-dimensional defects. 

Restricting our attention on the non-trivial part of perturbation we assume that $b$ is a cyclic vector for $U$, i.e.~that $\cH=\spa\{U^n b : n\in\Z\}$. In this case the operator $U_\gamma$, $|\gamma|<1$ is a completely non-unitary contraction, and thus unitarily equivalent to  its functional model $\cM_\gamma$, which is the compression of the multiplication by the independent variable $z$ onto the model space $\cK_{\theta_\gamma}$, where $\theta_\gamma$ here is the characteristic function of the contraction $U_\gamma$. 

The Clark operator $\Phi_\gamma$ is a unitary operator intertwining the operator $U_\gamma$, $|\gamma|<1$ (in the spectral representation of the operator $U$) and its model $\cM_\gamma$, $\cM_\gamma \Phi_\gamma = \Phi_\gamma U_\gamma$. In the case when the spectral measure of $U$ is purely singular (equivalently, the characteristic function $\theta_\gamma$ is inner) the operator $\Phi_\gamma$ was described from a slightly different point of view by D.~Clark \cite{Clark}. The case when $\theta_\gamma$ is an extreme point of the unit ball in $H^\infty$ was treated by D.~Sarason \cite{SAR} using the 
 {sub-Hardy spaces} $\cH(\theta)$ introduced by L.~de Branges. 

In this paper we treat the general case and give a systematic presentation of the subject. We first find a formula for the adjoint operator $\Phi^*_\gamma$ which is represented by a singular integral operator, generalizing in a sense the normalized Cauchy transform studied by A.~Poltoratskii, cf \cite{NONTAN}. We first present a ``universal'' representation that works for any transcription of the functional model. We then  give the formulas adapted for specific transcriptions of the model, such as Sz.-Nagy--Foia\c{s} and the de Branges--Rovnyak transcriptions, and finally obtain the representation of  $\Phi_\gamma$. 
\end{abstract}

\maketitle

\setcounter{tocdepth}{1}
\tableofcontents

%%%%%%%%%%%%%%%%%%%%%%%%%%%%%
\section{Introduction: main objects and aim of the paper}\label{s-intro}
%%%%%%%%%%%%%%%%%%%%%%%%%%%%%%%%%
\subsection{Rank one unitary perturbations}

For a unitary operator $U$ on a separable Hilbert space consider all its rank one perturbations $U+K$, $\rank K=1$ which are unitary. Writing 
\[
U+K= ( \I + KU^{-1})U
\]
we reduce the question to description of unitary rank one perturbations of the identity, which is an easy exercise. Namely, all possible rank one unitary perturbations of the identity are described as 
\[
\I + (\gamma-1) (\fdot, b)b = \I + (\gamma -1) bb^*, \qquad b\in\cH, \ \|b\|=1, \quad \gamma \in \T;
\]
here we use the ``linear algebra'' notation where a vector $b\in \cH$ is identified with the operator $\lambda\mapsto \lambda b$ acting from $\C$ to $\cH$, and the $b^*$ is the dual of this operator, i.e.~the linear functional $b^*f = (f, b)$. We will use this notation throughout the paper.

Using this representation one can immediately see that if one fixes the range of $K$, then all rank one perturbations $U+K$ are parametrized by the scalar parameter $\gamma\in \T$, namely 
\begin{align}
\label{f-Ugamma}
U_\gamma=\bigl(\I+ (\gamma-1)bb^* \bigr)U  = U+ (\gamma-1) b b_1^*, 
\end{align}
where $b \in\Ran K$, $\| b \|\ci\cH =1$, $b_1:=U^{-1}b$.  
Note that in this notation $U=U_1$. 

Since nothing interesting happens on the orthogonal complement $\spa\{U^k b: k\in\Z\}^\perp$, we assume 
without loss of generality that $b$ is a cyclic vector for $U$, i.e.~$\cH =%\overline
{ \spa\{U^k b: k\in\Z\}}$. 
It is an easy exercise to show that in this case $b$ is a cyclic vector for all $U_\gamma$, $\gamma \in \T$.

According to the Spectral Theorem there exists a unique Borel measure on $\T$ such that
\begin{align}
\label{SpectrMeas}
(U^nb, b) = \int_\T \xi^n d\mu(\xi), \qquad \forall n\in \Z, 
\end{align}
and $U=U_1$ is unitarily equivalent to the multiplication $M_\xi$ by the independent variable  $\xi$ on $L^2(\mu)$. So without loss of generality we assume that $U=U_1$ \emph{is} the multiplication operator $M_\xi$ on $L^2(\mu)$.  Identity \eqref{SpectrMeas} implies $\mu(\T)=1$ and that in this representation
\[
b(\xi)\equiv 1, 
\]
so 
\[
b_1(\xi) = \overline\xi. 
\]

With \eqref{f-Ugamma} it easily follows that for $|\gamma|<1$ the operators $U_\gamma$ are not unitary,  but contractive operators. 
Moreover, they are  what is called \emph{completely non-unitary} contractions, meaning that there is no  reducing subspace (i.e.~invariant for both $U_\gamma$ and $U_\gamma^*$) on which $U_\gamma$ acts unitarily. 

It is not difficult to compute the so-called \emph{defect operators} for $U_\gamma$, $|\gamma|<1$. Recall that for a contraction $T$ its defect operators $D_T$ and $D_{T^*}$ are defined as 
\begin{equation}
\label{defect}
D_T = \left( \I - T^*T \right)^{1/2}, \qquad D_{T^*} := \left( \I- TT^*\right)^{1/2}. 
\end{equation}
For $U_\gamma$, $|\gamma|<1$  defined above, we almost immediately obtain
\begin{align}
\label{defect-U_gamma}
D_{U_\gamma} =\left( 1 -|\gamma|^2\right)^{1/2} b_1b_1^*, \qquad  D_{U_\gamma^*} =\left( 1 -|\gamma|^2\right)^{1/2} b b^*.
\end{align}

\subsection{Functional models}
It is well-known that up to unitary equivalence a completely non-unitary contraction $T$ is fully determined  by its \emph{characteristic function} $\theta$ (see the definition below).  Namely,  $T$ is unitarily equivalent to its \emph{functional model} $\cM=\cM_\theta$, where $\cM_\theta$ is a \emph{compression}  of the multiplication operator $M_z$, 
\begin{align}
\label{model-01}
\cM_\theta  = P_{\theta} M_z \bigm|_{\textstyle \cK_\theta};
\end{align}
here $\cK_\theta $ is a subspace of a generally vector-valued, and possibly weighted $L^2$ space on the unit circle, $P_\theta = P_{\cK_\theta}$ is the orthogonal projection onto $\cK_\theta$, and $M_z$ is the multiplication by the independent variable $z$, $M_z f(z) = z f(z)$, $z\in \T$. 

There are several accepted models (transcriptions): the  Sz.-Nagy--Foia\c{s} model is probably the most known; the other two widely used ones are the de Branges--Rovnyak and  Pavlov models. 

While it is not relevant in the present paper, let us mention that according to Nikolski--Vasyunin \cite{Nik-Vas_model_MSRI_1998} all such functional models  can be obtained from the so-called \emph{free function model} by choosing a spectral representation of the minimal unitary dilation and then choosing the so-called \emph{functional embeddings}. 
A reader interested in the details should consult \cite{Nik-Vas_model_MSRI_1998}.  

We will only use the fact that a model is defined as above in \eqref{model-01}, our representation formula (Theorem \ref{t-reprB}) holds for any such a model.  We will also present the detailed calculations for the Sz.-Nagy--Foia\c{s} and for the de Branges--Rovnyak models.

\subsection{What is done in  the paper} In our case the defect subspaces of the operator $U_\gamma$, $|\gamma|<1$ are one dimensional, so the characteristic function $\theta_\gamma := \theta\ci{U_\gamma}$ is a (scalar-valued) bounded analytic function, and a model subspace $\cK_{\theta_\gamma}$ is a subspace of $L^2$ with values in $\C^2$ (with matrix weight in some transcriptions). 

The goal of this paper is to describe unitary operators $\Phi=\Phi_\gamma: \cK_{\theta_\gamma}\to L^2(\mu)$ intertwining the operator $U_\gamma$, $|\gamma|<1$ and 
its model, i.e.~such that $U_\gamma\Phi_\gamma = \Phi_\gamma \cM_{\theta_\gamma}$. Of course, such an operator depends on the transcription of the model, 
but to avoid overloading formulas we will not incorporate the transcription into the notation (although we will 
always say what transcription we are considering in the statements).
%%!!!! Do not like the wording much, but did not have anything better. 

In the \emph{classical} case of purely singular measures $\mu$ (which is equivalent to the characteristic function $\theta_\gamma$ being inner) the operators $\Phi_\gamma$ (in Sz.-Nagy--Foia\c{s} transcription) were described by Clark \cite{Clark}, so we will call $\Phi_\gamma$ the Clark operator. 

We will describe this operator and its adjoint for the general case. The adjoint Clark operator will be represented as a singular integral operator with Cauchy type kernel. 

The plan of the paper is as follows. 

In Section \ref{s:prelim} we give necessary information about the functional models, compute the characteristic functions for $U_\gamma$, and discuss the uniqueness of the Clark operator. Most of the material in this section is not new and is presented only for the convenience of the reader. 

In Section \ref{s:UnivRepr} we give a ``universal'' representation formula for the adjoint Clark operator $\Phi_\gamma^*$. ``Universal'' here means that the formula works for any transcription of the functional model. 

In Section \ref{s:AltRepr} we present a formula for the adjoint Clark operator adapted to the Sz.-Nagy--Foia\c{s} transcription. A feature of this transcription is that the ``top'' entry in the model space belongs to the Hardy space $H^2$. We present a formula where the ``top'' entry is given, as in the classical case (see Theorem 5.1 of \cite{Clark}) in Clark theory by the boundary values of the so-called \emph{normalized Cauchy transform}. 

Section \ref{s:Phi^*_gamma-deBranges} deals with a formula for $\Phi_\gamma^*$ in the de Branges--Rovnyak transcription for the general case. 
%(not restricted to the extreme case as in (III-8) of \cite{SAR}). 
In this transcription the elements of the model space are (boundary values of the) pairs of functions analytic inside and outside of the unit disc, and our representation formula is given in terms of the normalized Cauchy transforms inside and outside the unit disc.

Also in Section \ref{s:Phi^*_gamma-deBranges}, this representation is then used to get the formulas for $\Phi_\gamma$, which generalize Poltoratski's theorem \cite{NONTAN}.

In Section \ref{s:boundNCT} we improve a result concerning the boundedness of the normalized Cauchy transform as an operator. When taking the limits of the same expression from the outside of the unit disk, we discover that the resulting operator is an unbounded. We introduce an expression which yields a bounded operator that corresponds to the exterior Cauchy transform.

In Section \ref{s:OtherClarkMeas} we will revisit some of the previous results and explain how to treat the case of other Clark measures, i.e.~how to write the adjoint Clark operator in the spectral representation of the operator $U_\alpha$, $\alpha\in\T$. The main idea is rather straightforward, but the details turned out to be somewhat tedious. So for the readers' convenience we present the formulas with the detailed explanation. We relate our results to Clark's work \cite{Clark}.

In the appendix, Section \ref{s:V_alpha},  we show that the spectral representation of $U_\alpha$ is given by a singular integral operator of special form. Further we prove a certain converse statement: Given such an integral operator, we derive a rank one perturbation setup. In the related setting of  self-adjoint rank one perturbations these results were first obtained by the authors in \cite{mypaper}.

%%%%%%%%%%%%%%%%%%%%%%%%%%%%%%%%%%%%%%%%%%%%%%%%%%%
\section{Preliminaries: characteristic function and model for the operators 
%$U_\gamma$ 
\texorpdfstring{$U_\gamma$}{U<sub>gamma}
}\label{s:prelim}
%%%%%%%%%%%%%%%%%%%%%%%%%%%%%%%%%%%%%%%%%%%%%%%%%%%

Let us recall the main definitions. For a contraction $T$ acting on a separable Hilbert space let
\begin{align}
\label{DefectSpaces}
\fD=\fD_{T} := \clos \Ran D_T, \qquad \fD_*=\fD_{T^*} := \clos\Ran D_{T^*}
\end{align}
be the \emph{defect subspaces} of the operator $T$. The characteristic function $\theta=\theta_T$ of the operator $T$ is an analytic function $\theta =\theta_T \in H^\infty_{\fD\shto \fD_*}$ whose values are bounded operators (in fact, contractions) acting from $\fD$ to $\fD^*$ defined by the equation 
\begin{align}
\label{CharFunction}
\theta_T(z) = \left( -T + z D_{T^*}(\I - z T^*)^{-1} D_T \right) \Bigm|_{\textstyle\fD}, \qquad z\in \D. 
\end{align}
Note that $T\fD \subset \fD_*$, so for $z\in D$ the above expression indeed can be interpreted as an operator from $\fD$ to $\fD_*$. 

\begin{rem}
The following facts about the characteristic function $\theta =\theta\ci T$ are well known, cf \cite[Ch.~VI, s.~1]{SzNF2010}:
\begin{enumerate}
	\item $\theta\in H^\infty\ci{\fD\shto\fD_*}$, $\|\theta\|_\infty\le 1$;
	\item the function $\theta$ is \emph{purely contractive}, meaning that 
	\[
	\|\theta (0) x\| <\|x\| ,  \qquad \forall x\in \fD. 
	\]
\end{enumerate}
\end{rem}
\begin{rem}
The values $\theta(z)$ are contractions, so the non-tangential boundary values (limits in strong operator topology) exist a.e.~on $\T$. Some care is needed to  prove the existence of the non-tangential boundary values in the general case when $\fD$ and $\fD_*$ can be infinite-dimensional, but in our case $\dim \fD =\dim\fD_* =1$,  so the existence of boundary values is the standard fact from the theory of (scalar-valued) Hardy spaces. 
\end{rem}

We will follow the standard convention that the characteristic function is defined up to constant unitary factors on the right and on 
the left, i.e.~we consider the whole equivalence class consisting of functions $U\theta V$, 
where $U:\fD_*\to E_*$ and $V:\fD\to E$ are unitary operators and $E_*$, $E$ are Hilbert spaces of appropriate dimensions. 
The advantage of this point of view is that we are not restricted to using the defect spaces of $T$, but can work arbitrary Hilbert spaces of appropriate dimensions. 

Note, that the characteristic function (defined up to constant unitary factors) is a unitary invariant of a completely 
non-unitary contraction: any two such contractions with the same characteristic function are unitary equivalent.

Note also , that given a characteristic function, any representative  gives us a model, and there is a standard unitary equivalence between the model for different representatives.

\begin{rem*}
Another way to look at a choice of a representation of a characteristic function is to pick orthonormal bases in the defect spaces and treat the characteristic function as a matrix-valued function (possibly of infinite size). The choice of the orthonormal bases is equivalent to the choice of the constant unitary factors. 
\end{rem*}

\subsection{Functional models}
In this paper by a \emph{functional model} associated to an operator-valued function $\theta\in H^\infty_{E\shto E_*}$ we understand the following: a model space $\cK_\theta$ is an appropriately constructed subspace of a (possibly) weighted space $L^2(E_*\oplus E,W)$ on the unit circle $\T$ with the operator-valued weight  $W$.  The model operator $\cM_\theta$ is a compression of the multiplication operator $M_z$ onto $\cK_\theta$, 
\begin{align}
\label{model-02}
\cM_\theta  = P_{\theta} M_z \bigm|_{\textstyle \cK_\theta};
\end{align}
where $P_\theta = P_{\cK_\theta}$ is the orthogonal projection onto $\cK_\theta$. 

All the functional models for the same $\theta$ are unitarily equivalent, so sometimes people interpret them as different \emph{transcriptions} of one object. 

As we already mentioned above, a completely non-unitary contraction with characteristic function $\theta$ is unitarily equivalent to its model $\cM_\theta$. 

On the other hand, for any purely contractive $\theta\in H^\infty\ci{E\shto E_*}$, 
$\|\theta\|_\infty\le 1$ the model operator $\cM_\theta$ is a purely non-unitary contraction, with $\theta $ being its characteristic function. Thus, any such $\theta$ is a characteristic function of a completely non-unitary contraction.

%Let us also mention that for $\theta\in 

\subsection{Sz.-Nagy--%Foia\c{s}
\texorpdfstring{Foia\c{s}}{Foias} 
transcription}

The Sz.-Nagy--Foia\c{s} model (transcription) is probably the most used.

The model  space $\K_\theta$ is defined as a subspace of $L^2(E_*\oplus E)$ (non-weighted, $W(z)\equiv \I$), 
\begin{align}\label{K_theta}
 \K_\te=
 \left( \begin{array}{c} H^2_{E_*} \\ \clos\Delta L^2_{E} \end{array}\right) 
 \ominus 
  \left( \begin{array}{c} \theta \\ \Delta \end{array}\right) H^2_{E}
\intertext{where the defect $\Delta$ is given by}
 \Delta(z) := (1-\theta(z)^*\theta(z))^{1/2} , \qquad z\in \T
 .
\end{align}
If the characteristic function $\theta$ is \emph{inner}, meaning that its boundary values are isometries a.e.~on $\T$, then $\Delta \equiv 0$, so the lower ``floor'' of $\cK_\theta$ collapses and we get a simpler, ``one-story'' model subspace, 
\begin{align}
\label{oneStory-K_theta}
\cK_\theta = H^2(E_*)\ominus \theta H^2(E). 
\end{align}
This subspace is probably much more familiar to analysts, especially when $\theta$ is a scalar-valued function. 

The model operator $\cM$ is defined by \eqref{model-02} as the compression of the multiplication operator $M_z$ (also known as forward shift operator) onto $\cK_\theta$, 
and the multiplication operator $M_z$ is understood as the entry-wise multiplication by the independent variable $z$, 
\[
M_z \left(\begin{array}{c} g \\ h \end{array} \right) = \left(\begin{array}{c} zg \\ zh \end{array} \right). 
\]

As we discussed above, 
the characteristic function $\theta$ is defined up to constant unitary factors on the right and on the left. 
But one has to be a bit careful here, because if $\wt \theta(z) = U \theta (z) V$, where $U$ and $V$ are constant unitary operators, then the spaces $\cK_\theta$ and $\cK_{\wt\theta}$ are different. 

However, the map $\cU$
\[
\cU \left(\begin{array}{c} g \\ h \end{array} \right) = \left(\begin{array}{c} U g \\ V^*  h \end{array} \right) 
\]
is the canonical unitary map transferring the model from one space to the other. 

Namely, it is easy to see that $\cU$ is a unitary map from $H^2(E_*) \oplus \clos\Delta L^2 (E)$ onto $H^2(UE_*) \oplus \clos \wt\Delta L^2 (V^* E)$, where $\wt \Delta = \Delta_{\wt \theta} = V^*\Delta V$. Moreover, it is not difficult to see that $\cU \cK_\theta = \cK_{\wt\theta}$ and that $\cU$ commutes with the multiplication by $z$, so $\cU_\theta := \cU\bigm|_{\textstyle \cK_\theta}$ intertwines the model operators, 
\[
\cU_\theta \cM_\theta = \cM_{\wt\theta} \cU_\theta  .
\]

\subsection{de Branges--Rovnyak transcription}
\label{s:deBrangesRepr}
Let us present this transcription as it is described in \cite{Nik-Vas_model_MSRI_1998}. Since the ambient space in this transcription is a weighted $L^2$ space with an operator-valued weight, let us recall that if $W$ is an operator-valued weight on the circle, i.e.~a function whose values are self-adjoint non-negative operators in a Hilbert space $E$, then the norm in the space $L^2(W) $ is defined as 
\[
\|f\|_{L^2(W)} = \int_\T \left( W(z) f(z), f(z)\right)\ci{E} \frac{|dz|}{2\pi} \,.
\]
There are some delicate details here in defining the above integral if we allow the values $W(z)$ to be unbounded operators, but we will not discuss it here. In our case when the characteristic function is scalar-valued the values $W(z)$ are bounded self-adjoint operators on $\C^2$, and the definition of the integral is straightforward. 

Let 
\[
W_\theta(z) = \left(\begin{array}{cc} \I & \theta(z) \\ \theta(z)^* & \I \end{array}\right)\, .
\]
The weight in the ambient space will be given by $W=W_\theta^{[-1]}$, $W_\theta^{[-1]}(z) = (W_\theta(z))^{[-1]}$ where 
$A^{[-1]}$ stands for the Moore--Penrose inverse of the operator $A$. If $A=A^*$ then $A^{[-1]}$ is $\OZ$ on $\Ker A$ and is 
equal to the left inverse of $A$ on $\Ran A$. The model space
$\cK_\theta$ is defined as
\begin{align}
\label{deBrangesRepr}
\cK_\theta = \left\{
\left(\begin{array}{c} g_+ \\ g_-\end{array} \right) \,:\ g_+\in H^2(E_*),\  g_- \in H^2_-(E),\  g_- - \theta^* g_+ \in \Delta L^2(E)
\right\}   .
\end{align}

\begin{rem}
\label{r:dBr-R_orig}
The original de Branges--Rovnyak model was initially described in \cite{deBr-Rovn_CanonModels_1966} in  completely different terms. To give the definition from \cite{deBr-Rovn_CanonModels_1966} we need to recall the notion of a Toeplitz operator. For $\f\in L^\infty_{E\shto E_*}$ the Toeplitz operator $T_\f : H^2(E)\to H^2(E_*)$ with symbol $\f$ is defined by 
\[
T_\f f := P_+ (\f f), \qquad f\in H^2(E). 
\]

The (preliminary) space $\cH(\theta) \subset H^2(E_*)$ is defined as a range $(\I - T_{\theta}T_{\theta^*})^{1/2} H^2 (E)$ endowed with the \emph{range norm} (the minimal norm of the preimage). 

Let the involution operator $J$ on $L^2(\T)$ be defined as 
\[
Jf(z) = \overline z f(\overline z). 
\]
Following de Branges--Rovnyak \cite{deBr-Rovn_CanonModels_1966} define  the \emph{model space} $\cD(\theta)$  as the set of vectors 
\[
\left( \begin{array}{c}
g_1 \\ 
g_2 \\ 
\end{array} \right) \ : g_1\in \cH(\theta), \ g_2\in H^2(E),  \text{ such that } z^n g_1 -\theta P_+(z^n Jg_2) \in \cH(\theta) \ \forall n\ge 0, 
\] 
and such that 
\[
\left\| \left( \begin{array}{c}
g_1 \\ 
g_2 \\ 
\end{array} \right) \right\|_{\cD(\theta)}^2 := \lim_{n\to\infty}
\left( \|z^n g_1 -\theta P_+(z^n Jg_2)\|\ci{\cH(\theta)}^2 + \|P_+(z^n Jg_2)\|_2^2 \right) <\infty. 
\]
It might look surprising, but it was proved in \cite{Nik-Vas_FunctModels_1989} that the operator 
$\left( \begin{array}{c}
g_+ \\ 
g_- \\ 
\end{array} \right)
\mapsto 
\left( \begin{array}{c}
g_+ \\ 
Jg_- \\ 
\end{array} \right)
$
is a unitary operator between the described above model space $\cK_\theta$ in the de Branges--Rovnyak transcription  and the model space $\cD(\theta)$. 
\end{rem}

\subsection{Characteristic function for %$U_\gamma$%
\texorpdfstring{$U_\gamma$}{U<sub>gamma}
}
\label{s:CharFunctU_gamma}
One immediately sees from \eqref{defect-U_gamma} that the defect spaces $\fD_{U_\gamma}$ and $\fD_{U^*_\gamma}$ are spanned by the vectors $b_1$ and $b$ respectively. Pick these vectors as orthonormal bases in the corresponding defect spaces. The only freedom we have in the choice of the bases  is to multiply the vectors $b_1$ and $b$ by unimodular constant, and we set both these constants to be $1$. 

For a (possibly complex-valued) measure $\tau$ on $\T$ and $\la\notin\T$ define the following Cauchy type transforms $R$, $R_1$ and $R_2$
\begin{align}
\label{CauchyTrans}
R \tau (\la) := \int_\T \frac{d\tau(\xi)}{1-\overline\xi\la}, \qquad R_1 \tau (\la) := \int_\T \frac{\overline\xi \la d\tau(\xi)}{1-\overline\xi\la}, \qquad R_2\tau(\la):= \int_\T \frac{1+ \overline\xi \la }{1-\overline\xi\la} d\tau(\xi).
\end{align}

Recall that $U=U_1$ is the multiplication by the independent variable $\xi$ in $L^2(\mu)$, and the operators $U_\gamma$ are defined as $U_\gamma = U+ (\gamma-1) b b_1^*$, where $b(\xi)\equiv 1$, $b_1(\xi) \equiv \overline \xi$. 

\begin{lem}
\label{l:CharFunct-U_0}
Under the above assumption the characteristic function $\theta =\theta_0$ of the operator $U_0$ (with respect to the bases $b_1$ and $b$ in the defect spaces) is given by 
\[
\theta(\la) =  \frac{R_1 \mu(\la)}{1 + R_1 \mu(\la)} =  \frac{R_2\mu(\la) -1}{R_2\mu(\la)+1}, \qquad \la\in \D.
\]
\end{lem}

\begin{lem}
\label{CharFunct-U_gamma}
Under the above assumption the characteristic function $\theta_\gamma$ of the operator $U_\gamma$, $|\gamma|<1$ (with respect to the bases $b_1$ and $b$ in the defect spaces) is given by 
\[
\theta_\gamma(\la) =  -\gamma + \frac{(1-|\gamma|^2) R_1 \mu(\la)}{1 + (1-\overline\gamma) R_1 \mu(\la)} = 
\frac{(1-\gamma) R_2\mu(\la) -(1+\gamma)}{(1-\overline\gamma )R_2\mu(\la)+(1+\overline\gamma)}, \qquad \la\in \D.
\]
\end{lem} 

For the proofs of the latter two lemmata we refer the reader to Theorem 5.1 of \cite{RonRkN}. Notice that their choice of bases for the defect spaces agrees with ours. In other words, we consider the same representative from the equivalence class of characteristic functions.

For finite rank perturbations an analogous formula for the characteristic function was also obtained in Theorem 4.1 of \cite{RonRkN}.

\begin{rem}
\label{l:theta_0-theta_gamma}
It can be easily seen by the direct calculations that 
\begin{align}
\label{theta_0-theta_gamma}
\theta_\gamma = \frac{\theta_0-\gamma}{1-\overline\gamma\theta_0} \qquad
\textup{or equivalently } \qquad \theta_0 = \frac{\theta_\gamma + \gamma}{1 + \overline\gamma\theta_\gamma} ,
\end{align}
and that
\begin{align*}
\Delta_\gamma = \frac{(1- |\gamma|^2)^{1/2}}{|1-\overline\gamma\theta_0|} \Delta_0.
\end{align*}
\end{rem}

\subsection{Defect subspaces for the model operator in the Sz.~Nagy--%
%Foia\c{s}
\texorpdfstring{Foia\c{s}}{Foias} 
transcription}
The following simple and well-known lemma describes the defect spaces of a model operator.

\begin{lem}
\label{l:DefectsModelOp}
Let $\theta\in H^\infty$ (scalar-valued), $\|\theta\|_\infty\le 1$ be a purely contractive function, 
and let $\cM=\cM_\theta$ be the corresponding model operator, $\cM= P_\theta M_z \bigm|_{\textstyle \cK_\theta}$. Then the defect spaces of $\cM$ are given by
\begin{align}
\label{defects}
\fD_{\cM^*} = \spa\{c\}, \qquad \fD_{\cM} = \spa\{ c_1\}, 
\end{align}
where 
\begin{align}
\label{c}
c(z) &:= \left( 1- |\theta(0)|^2 \right)^{-1/2} 
\left( \begin{array}{c} 1-\overline{\theta(0)}\theta (z)\\ -\overline{\theta(0)}\Delta (z) \end{array} \right), 
\\
\label{c_1}
c_1(z) & := \left( 1- |\theta(0)|^2 \right)^{-1/2} 
\left( \begin{array}{c} z^{-1} \left(\theta(z)- \theta(0)\right) \\ z^{-1} \Delta (z)\end{array} \right), 
\end{align}
$\|c\|=\|c_1\|=1$. 
Moreover, 
\begin{align}
\label{Mc_1}
\cM c_1 = -\theta(0) c. 
\end{align}
\end{lem}

\begin{proof}
Recall that $\cM = P_\te M_z$.
If $\te(0)\neq 0$, then $\cM^\ast = P_\te S^\ast$, where
\[
S^\ast \kf{g_1}{g_2} = \kf{z^{-1} (g_1(z)-g_1(0))}{z^{-1} g_2(z)},\qquad \kf{g_1}{g_2}\in\cK_\te.
\]

With this, we have
\[
 (I - \cM \cM^\ast)  \kf{g_1}{g_2} = \kf{g_1}{g_2} - P_\te\kf{g_1(z) -g_1(0)}{g_2} = P_\te \kf{g_1(0)}{0}.
 \]
 For the second equality we used that $\kf{g_1}{g_2}\in \cK_\te$ implies $\kf{g_1}{g_2}\perp \kf{\te}{\Delta}H^2$. Since the range must be non-trivial, we can assume without loss of generality that $g_1(0) = 1$. We obtain
 \begin{align}\label{equation}
 (I - \cM \cM^\ast)  \kf{g_1}{g_2} = P_\te \kf{1}{0}  = \left( \begin{array}{c} 1-\overline{\theta(0)}\theta (z)\\ -\overline{\theta(0)}\Delta (z) \end{array} \right),
\end{align}
and the first equation of \eqref{defects}.

The normalizations follow from straightforward computations.

To see the second equation of \eqref{defects}, we simply compute
\[
 (I - \cM^\ast \cM)  \kf{g_1}{g_2} = S^\ast \kf{\te}{\Delta} P_+(\bar\te z g_1 + \Delta z g_2) = \kf{z^{-1}(\te(z)-\te(0))}{z^{-1}\Delta(z)} P_+(\bar\te z g_1 + \Delta z g_2).
\]

It remains to prove \eqref{Mc_1}. We have
\[
\cM c_1 = P_\te M_z c_1 = \left( 1- |\theta(0)|^2 \right)^{-1/2} 
P_\te \left( \begin{array}{c} \theta(z)- \theta(0) \\ \Delta (z)\end{array} \right).
\]
By the definition of the Sz.-Nagy--Foia\c s space \eqref{K_theta} the projection
$$
P_\te \left( \begin{array}{c} \theta(z) \\ \Delta (z)\end{array} \right) = {\bf 0},
$$
the zero vector of $\cK_\te$. And recalling equation \eqref{equation}, we get
\[
P_\te \left( \begin{array}{c}  -\theta(0) \\ 0\end{array} \right)
=
-\theta(0) P_\te \left( \begin{array}{c}  1 \\ 0\end{array} \right)
=
-\te(0) \left( \begin{array}{c} 1-\overline{\theta(0)}\theta (z)\\ -\overline{\theta(0)}\Delta (z) \end{array} \right).
\]
Summing up and recalling the definition \eqref{c} of $c$ we conclude \eqref{Mc_1}.
\end{proof}

\subsection{Freedom in the choice of the Clark operator 
%$\Phi_\gamma$
\texorpdfstring{$\Phi_\gamma$}{Phi<sub>gamma}
}
Assume that a transcription for the model is fixed. 

Let $\theta_\gamma$ be a characteristic function (one of the representatives) of the operator $U_\gamma$, $|\gamma|<1$. We want to obtain a representation for the adjoint of a Clark operator, i.e.~for a unitary operator $\Phi_\gamma^*: L^2(\mu) \to \cK_{\theta_\gamma}$, 
\begin{align}
\label{Phi*-intertwine}
\Phi_\gamma^* U_\gamma = \cM_{\theta_\gamma} \Phi_\gamma^*. 
\end{align}
Note that such a unitary operator is not unique, one can, for example, multiply it by a unimodular constant. The proposition below says that with the exception of the case when $\gamma=0$ and $\mu$ is the normalized Lebesgue measure that is the only degree of freedom we have.

A unitary operator $\Phi^*_\gamma$ satisfying \eqref{Phi*-intertwine} maps $\fD_{U_\gamma}$ onto $\fD_{\cM_{\theta_\gamma}}$, and $\fD_{U_\gamma^*}$ onto $\fD_{\cM_{\theta_\gamma}^*}$. Take a unit vector in one of the defect subspaces, for example take $c\in \fD_{\cM_{\theta_\gamma}^*}$, $\|c\|=1$. Multiplying, if necessary, $\Phi^*_\gamma$ by an appropriate unimodular constant we find $\Phi_\gamma^*$ satisfying \eqref{Phi*-intertwine} such that $\Phi_\gamma^* b =c$. 

\begin{defn}
\label{d:c-agree}
We say that the unit vectors $c\in \fD_{\cM_{\theta_\gamma}^*}$ and $c_1\in \fD_{\cM_{\theta_\gamma}}$, $\|c\|=\|c_1\|=1$ \emph{agree} if there exists a unitary map $\Phi_\gamma^*: L^2(\mu)\to \cK_{\theta_\gamma}$ satisfying \eqref{Phi*-intertwine} such that 
\[
\Phi_\gamma^* b = c, \qquad \Phi_\gamma^* b_1 =c_1 .
\]
\end{defn}

Note that multiplying $\Phi^*_\gamma$ by a unimodular constant, if necessary, we can always assume without loss of generality that $\Phi_\gamma^* b=c$. 

\begin{prop}
\label{p:c-agree}
Let $\theta_\gamma$ be a characteristic function (one of the representatives) of the operator $U_\gamma$, $|\gamma|<1$. 
If $\gamma = 0$ and $\mu$ is the Lebesgue measure (so $\theta_\gamma\equiv 0$), then any unit vectors $c\in \fD_{\cM_{\theta_\gamma}^*}$ and $c_1\in \fD_{\cM_{\theta_\gamma}}$, $\|c\|=\|c_1\|=1$ agree. 

In all other cases given $c\in \fD_{\cM_{\theta_\gamma}^*}$, $\|c\|=1$ there exists a unique vector $c_1 \in \fD_{\cM_{\theta_\gamma}}$ such that $c$ and $c_1$ agree (of course, in this case $\|c_1\|=1$). 

Moreover
\begin{enumerate}
\item If $\gamma\ne 0$ and $\Phi_\gamma^* b=c$, the vector $c_1$ can be found as the unique vector in the one-dimensional subspace $\fD_{\cM_{\theta_\gamma}}$ satisfying the equation
\begin{align*}
\label{find:c_1-01}
(c_1, \cM_{\theta\gamma}^* c) = (\cM_{\theta\gamma} c_1,  c) = \gamma. 
\end{align*}

\item If $\mu$ is not the Lebesgue measure and $\Phi_\gamma^* b=c$, the vector $c_1$ can be found as the unique vector in the one-dimensional subspace $\fD_{\cM_{\theta_\gamma}}$ satisfying the equation 
\[
(c_1, \cM_{\theta_\gamma}^n c)\ci{\cK_{\theta_\gamma}} = \widehat \mu(n+1), 
\]
where $n= \min\{k\in\Z_+: \widehat \mu(k+1) \ne 0\}$. 
\end{enumerate}

\end{prop}
\begin{proof}
The existence of $c_1$ follows from the model theory (existence of an intertwining unitary operator $\Phi_\gamma^*$ satisfying \eqref{Phi*-intertwine}). Multiplying $\Phi^*_\gamma$ by a unimodular constant, if necessary, we can assume without loss of generality that $\Phi_\gamma^* b=c$. We then define $c_1 :=\Phi^*_\gamma b_1$; because of the intertwining relation \eqref{Phi*-intertwine} $\Phi_\gamma^*$ maps the defect subspaces of $U_\gamma$ to the corresponding defect subspaces of $\cM_{\theta_\gamma}$, so $c_1\in\fD_{\cM_{\theta_\gamma}}$.

If $\gamma=0$ and $\mu $ is the Lebesgue measure, Lemma \ref{l:CharFunct-U_0} asserts that the characteristic function $\theta\equiv 0$. The model space $\cK_\theta$ is represented in this case in the Sz.-Nagy--Foia\c{s} transcription as 
\[
\cK_\theta = \{ g_+ \oplus g_- : g_+ \in H^2, \ g_-\in H^2_- \} . 
\]
The defect subspaces $\fD_{\cM_{\theta_\gamma}}$ and $\fD_{\cM_{\theta_\gamma}^*}$ are spanned by the vectors $c_1 = 0\oplus z^{-1}$ and $c=\1 \oplus 0$ respectively. 

If $\mu$ is the Lebesgue measure, then the system of exponents $e_n(\xi) = \xi^n$, $n\in\Z$ forms an orthonormal basis in $L^2(\mu)$, and it is easy to see that for any $\alpha, \beta \in\T$ the operator $\Phi_0^*$
\[
\Phi_0^* e_n =
\left\{
\begin{array}{ll} 
\alpha z^n \oplus 0, \qquad & n\ge 0, \\
 0\oplus \beta z^n,   & n< 0
\end{array}
\right. 
\]
satisfies \eqref{Phi*-intertwine} with $\gamma=0$. Then for this $\Phi_\gamma^*$ we have 
\[
\Phi_\gamma^* b = \alpha c, \qquad \Phi^*_\gamma b_1 = \beta c_1. 
\]

Let $\gamma\ne0$. Let $\Phi_\gamma^* :L^2(\mu)\to \cK_{\theta_\gamma}$ be a unitary operator satisfying \eqref{Phi*-intertwine} and such that $\Phi_\gamma^* b=c$ (as we discussed above, such an operator always exists). Since 
\[
U_\gamma^* b = \overline\gamma b_1
\]
the vector $c_1=\Phi_\gamma^* b_1$ must then satisfy 
\begin{align}
\label{find:c_1}
\cM_{\gamma_\theta}^* c=\overline\gamma c_1.
\end{align}
Because $\gamma\ne0$ vector $c_1$ is unique. Note, that since $\cM_{\theta_\gamma} c \in \fD_{\cM_{\theta_\gamma}^*}$, to check identity \eqref{find:c_1} it is enough to check that 
\[
(c_1, \cM_{\gamma_\theta}^* c) = (\cM_{\gamma_\theta} c_1,  c) = \gamma, 
\]
which is the equation in the statement \cond1 of the proposition.

Now, let $\mu$ be not the Lebesgue measure. 
Let $\Phi_\gamma^* b = c$ and let $c_1:=\Phi_\gamma^* b_1$ (clearly $c_1\in \fD_{\cM_{\theta_\gamma}}$ and $\|c_1\|=1$). 

Let $n\ge 0$ be smallest integer such that $(b_1, U_\gamma^n b) \ne 0$. 
To show that such  $n<\infty$ exists let us first note that if $x\perp b_1$ then 
\[
(b_1, U_\gamma x) = (b_1, Ux). 
\]
This implies that 
\begin{align}
\label{n-muHat}
n=\min\{k\in\Z_+: (b_1, U_\gamma^k b) \ne 0 \} = \min\{ k\in\Z_+ : (b_1, U^k b) \ne0\}. 
\end{align}
By the spectral theorem, the second representation means that 
\[
n = \min \{ k\in \Z_+ : \widehat \mu(k+1) \ne 0\}
\]
where $\widehat \mu(k) $ is the $k$th Fourier coefficient of the measure $\mu$. If $\widehat\mu(k+1) =0$ for all $k\in\Z_+$, then because $\mu$ is a real measure $\widehat\mu(k) =0$ for all $k\ne 0$, so the measure $\mu$ is the Lebesgue measure, which contradicts the assumption. So such $n<\infty$ exists. 

It also follows from the above reasoning that 
\[
(b_1, U_\gamma^n b) = (b_1, U^n b) = \widehat\mu(n+1) , 
\]
so $c_1$ must satisfy 
\[
(c_1, \cM_{\theta_\gamma}^n c) = \widehat \mu(n+1). 
\]
Since the defect subspace $\fD_{\cM_{\theta_\gamma}}$ is one-dimensional and $\widehat\mu(n+1)\ne0$ the above equation has unique solution $c_1\in \fD_{\cM_{\theta_\gamma}}$.
\end{proof}

The next proposition deals with the Sz.-Nagy--Foia\c{s} transcription. 
\begin{prop}
\label{p:Phi-c1}
Let the characteristic function be given by Lemma \ref{CharFunct-U_gamma}, and let $c^\gamma$ and $c_1^\gamma$ be defined by \eqref{c}, \eqref{c_1} with $\theta=\theta_\gamma$. Then the vectors $c^\gamma$ and $c_1^\gamma$ agree. 
\end{prop}

This proposition is trivial if $\gamma=0$ and $\mu$ is the Lebesgue measure, because  in this case, by Proposition \ref{p:c-agree} any two unit vectors in the corresponding defect spaces agree. 
 
\begin{proof}[Proof of Proposition \ref{p:Phi-c1}]
If $\gamma\ne 0$ it is sufficient to check statement \cond1 of Proposition \ref{p:c-agree}. Note that since $\cM_\theta f = P_\theta (zf)$ and $c\in\cK_\theta$, 
\[
(\cM_\theta c_1, c) = (z c_1, c). 
\]
Using identities \eqref{c}, \eqref{c_1} and recalling that $\theta=\theta_\gamma$ implies $\theta_\gamma(0)=-\gamma$, we obtain
\begin{align*}
(z c_1, c) & = (1-|\gamma|^2)^{-1} \bigl( 
\left(\begin{array}{c} \theta \\ \Delta \end{array} \right) + 
\left(\begin{array}{c} \gamma \\ 0 \end{array} \right), 
\left(\begin{array}{c} 1+ \overline\gamma\theta \\ \overline\gamma\Delta \end{array} \right)
\bigr)\ci{L^2_{\C^2}}
\\
& = 
(1-|\gamma|^2)^{-1} \bigl( 
\left(\begin{array}{c} \gamma \\ 0 \end{array} \right), 
\left(\begin{array}{c} 1+ \overline\gamma\theta \\ \overline\gamma\Delta \end{array} \right)
\bigr)\ci{L^2_{\C^2}}
\\
& = 
(1-|\gamma|^2)^{-1} \gamma \bigl(\1, \1+\overline\gamma \theta \bigr)\ci{L^2} = 
(1-|\gamma|^2)^{-1} \gamma (1+\gamma\overline{\theta(0)}) = \gamma;
\end{align*}
here the second equality holds because $\left(\begin{array}{c} \theta \\ \Delta \end{array} \right)  \perp\cK_\theta$, and in the last line we again used the fact that $\theta(0) =-\gamma$. 

If $\mu$ is not the Lebesgue measure we can check the identity from the statement \cond2 of Proposition \ref{p:c-agree}. 
It is a standard and a well known fact of the  functional model theory that 
\begin{align}
\label{dilation}
\cM_\theta^k = P_\theta M_z^k \bigm|_{\textstyle \K_\theta};
\end{align}
the ``high brow'' explanation is that all the functional models are constructed in such a way that the operator $M_z$ in the ambient space is a unitary dilation of the model operator $\cM_\theta$. In the
Sz.-Nagy--Foia\c{s} transcription one can check \eqref{dilation} directly: 
it  easily follows from the fact that $\cK_\theta \oplus \left( \begin{array}{c} \theta \\ \Delta \end{array}\right) H^2$ and  $\left( \begin{array}{c} \theta \\ \Delta \end{array}\right) H^2$ are $M_z$-invariant subspaces, we leave the details as an exercise for the reader.

Direct computations using \eqref{c}, \eqref{c_1} give us
\begin{align}
\label{c,z^nc_1}
(c_1, z^n c) = (1-|\gamma|^2)^{-1} \widehat \theta(n+1). 
\end{align}
Using the first representation for $\theta_\gamma$ (the one involving $R_1$) from Lemma \ref{CharFunct-U_gamma} we can easily see that 
\begin{align}
\label{hatTheta}
\widehat \theta_\gamma(n+1) = (1-|\gamma|^2) \widehat\mu(n+1). 
\end{align}
Indeed, 
\[
R_1\mu(z) = \sum_{k=1}^\infty \widehat \mu(k) z^k =\sum_{k=n+1}^\infty \widehat \mu(k) z^k, 
\]
so dividing the power series in the representation from Lemma \ref{CharFunct-U_gamma} we get \eqref{hatTheta}. 

Combining \eqref{c,z^nc_1} and \eqref{hatTheta} we obtain
\[
(c_1, \cM_{\theta_\gamma}^n c) = (c_1, z^n c) = \widehat\mu(n+1), 
\]
which is exactly the identity from the statement \cond2 of Proposition \ref{p:c-agree}.
\end{proof}

%%%%%%%%%%%%%%%%%%%%%%%%%%%%%%%%%%%%%%%%%%%%%%%%%%
\section{A 
%``universal''
\texorpdfstring{``}{"}universal'' 
representation formula for 
%$\Phi_\gamma^*$
\texorpdfstring{$\Phi_\gamma^*$}{Phi<sub>gamma<sup>*} 
and some of its properties}
\label{s:UnivRepr}

In the next theorem we fix an arbitrary transcription of the functional model. 

\begin{theo}[A ``universal'' representation formula]
\label{t-reprB}
Let $\theta_\gamma$  be a characteristic function (one representative) of $U_\gamma$,  $|\gamma|<1$, and let $\cK_{\theta_\gamma}$ and $\cM_\gamma =\cM_{\theta_\gamma}$ be the model subspace and the model operator respectively. 
Assume that the vectors $c=c^\gamma\in\fD_{\cM_{\theta_\gamma}^*}$, $c_1=c_1^\gamma \in \fD_{\cM_{\theta_\gamma}}$  $\|c^\gamma\|=\|c_1^\gamma\|=1$ agree. 
Let $\Phi^*=\Phi^*_\gamma: L^2(\mu) \to \cK_{\theta_\gamma}$ be a unitary operator satisfying 
\[
\Phi_\gamma^* U_\gamma = \cM_{\theta_\gamma} \Phi_\gamma^* , 
\]
and such that $\Phi_\gamma^* b = c^\gamma$, $\Phi_\gamma^* b_1 = c_1^\gamma$. 

Then for all $f\in C^1(\T)$
\begin{align}
\label{f-reprB}
\Phi_\gamma^* f(z)&= A_\gamma(z) f(z) +
B_\gamma(z)\int\frac{f(\xi)-f(z)}{1- \overline\xi z}\, d\mu(\xi)
\end{align}
where $A_\gamma(z) = c^\gamma(z)$, $B_\gamma(z) = c^\gamma(z) - z c_1^\gamma(z)$.
\end{theo}

\begin{rem}
Let the representation of $\theta_\gamma$ be as given in Lemma \ref{CharFunct-U_gamma}. 
Then by  Proposition \ref{p:Phi-c1} the vectors $c^\gamma$ and $c_1^\gamma$ given by \eqref{c} and \eqref{c_1} with $\theta=\theta_\gamma$ agree.

By Remark \ref{l:theta_0-theta_gamma}, we have $\te_\gamma(0) = -\gamma$, so
for the above $c^\gamma$ and $c_1^\gamma$ 
the functions $A_\gamma$ and $B_\gamma$ are given in the Sz.-Nagy--Foia\c{s} transcription (for $z\in\T$) by
\begin{align}
\label{A(z)}
A_\gamma(z) & = c^\gamma(z) = (1-|\gamma|^2)^{-1/2} 
\left(\begin{array}{c} 1 + \overline\gamma  \theta_\gamma(z) \\ \overline\gamma \Delta_\gamma(z) \end{array}\right) 
= \left(\begin{array}{c}  \frac{(1-|\gamma|^2)^{1/2}}{1 - \overline\gamma  \theta_0(z)} \\ \frac{\overline\gamma\Delta_0(z)}{| 1- \overline\gamma \theta_0(z)|} \end{array}\right) 
\\
\label{B(z)}
B_\gamma(z) & = c^\gamma(z) - z c_1^\gamma(z)  = (1-|\gamma|^2)^{-1/2} 
\left(\begin{array}{c} 1 + (\overline\gamma  -1)\theta_\gamma(z) -\gamma\\ (\overline\gamma -1)\Delta_\gamma(z) \end{array}\right)
\\
\notag
&  \qquad\qquad =
\left(\begin{array}{c}  (1-|\gamma|^2)^{1/2}
{(1-\theta_0(z))}/{(1 - \overline\gamma  \theta_0(z))} \\ (\overline\gamma -1) 
{\Delta_0(z)}/{| 1- \overline\gamma \theta_0(z)|} \end{array}\right) 
 . 
\end{align}
\end{rem}

\begin{proof}[Proof of Theorem \ref{f-reprB}]
First, let us notice that one can write the model operator $\cM_{\theta_\gamma} = P_{\theta_\gamma} M_z \bigm|_{\textstyle \cK_{\theta_\gamma}}$ without explicitly using the projection $P_{\theta_\gamma}$. Namely, denoting $c_2^\gamma(z) = zc_1^\gamma(z)$ we can see that on $\cK_{\theta_\gamma}$ 
\begin{align}
\label{M_theta-01}
\cM_{\theta_\gamma}  & = M_z -c_2^\gamma (c_1^\gamma)^* - \theta_\gamma(0) c^\gamma (c_1^\gamma)^*  \\
\notag
& = M_z + (\gamma c^\gamma - c_2^\gamma )(c_1^\gamma)^*  . 
\end{align}
Indeed, $\cM_\theta f = M_z f$ for $\cK_\theta\cap c_1^\perp$,   
%the model operator $\cM_\theta$ acts as $M_z$, 
so the formula holds for $f\in \cK_{\theta_\gamma}\cap (c_1^\gamma)^\perp$.   It is not hard to compute that $\cM_\theta c_1= \theta(0) c$, so, since $M_z c_1 - c_2 c_1^* c_1 =0$, the first line of \eqref{M_theta-01} holds for $f=c_1$. For the second line the equality $\gamma=-\theta(0)$ was used. 

Using the above formula we can rewrite identity $\Phi_\gamma^* U_\gamma = \cM_{\theta_\gamma} \Phi^*_\gamma$ as 
\begin{align*}
\Phi^*_\gamma U_1 + (\gamma-1) c^\gamma b_1^* = M_z \Phi_\gamma^* + (\gamma c^\gamma -c_2^\gamma) b_1^*
\end{align*}
(here we used that $\Phi^*_\gamma b =c^\gamma$ and $(c_1^\gamma)^* \Phi_\gamma^* = (\Phi_\gamma c_1^\gamma)^* = b_1^*$), so 
\begin{align}
\label{CommRel}
\Phi_\gamma^* U_1 = M_z \Phi_\gamma^*  + (c^\gamma-c_2^\gamma) b_1^*. 
\end{align}
%%
%To simplify the notation, let us skip in the next calculations the sub/superscript $\gamma$. 

Right multiplying \eqref{CommRel} by $U_1$ and using \eqref{CommRel} for  $\Phi^* U_1$ in $M_z\Phi^* U_1$ we get
\[
\Phi^*_\gamma U_1^2 = M_z^2 \Phi^*_\gamma + M_z (c^\gamma-c_2^\gamma) b_1^* + (c^\gamma-c^\gamma_2)(U_1^* b_1)^* . 
\]
Right multiplying the obtained identity by $U_1$ and repeating this procedure we get by induction that for all $n\ge 1$
\[
\Phi^*_\gamma U_1^n = M_z^n \Phi^*_\gamma + \sum_{k=1}^n M_z^{k-1} (c^\gamma-c_2^\gamma)\left( (U_1^*)^{n-k} b_1 \right)^* . 
\]
Let us now apply both sides of this formula to the vector $b\equiv \1 \in L^2(\mu)$.  Note that  
$ (U_1^*)^{n-k} b_1 \equiv \xi^{-(n-k+1)} $, so 
\[
\left( (U_1^*)^{n-k} b_1 \right)^*  b = \int_\T \xi^{n+k-1} d\mu(\xi). 
\]
Since $U_1^n b \equiv \xi^n$ and $\Phi^*_\gamma b = c^\gamma$ we get (using $\xi$ and $z$ for the independent variables in $L^2(\mu) $ and $\cK_\theta$ respectively) that
\[
\left( \Phi^*_\gamma \xi^n\right) (z) = z^n c^\gamma(z) + (c^\gamma(z) - c^\gamma_2(z) ) \int_\T \left(\sum_{k=1}^n z^{k-1}\xi^{n-k+1} \right) d\mu(\xi). 
\]
Summing the geometric progression we get 
\[
\sum_{k=1}^n z^{k-1}\xi^{n-k+1} = \frac{\xi^n - z^n}{1-\overline \xi z} , 
\]
so the representation formula \eqref{f-reprB} holds for any analytic monomial $f(\xi) = \xi^n$, $n\in\N$. Note that it also holds trivially for $f(\xi) \equiv 1$, so \eqref{f-reprB} holds for all (analytic) polynomials in $\xi$. 

Now, let us show that the representation formula \eqref{f-reprB} holds for all trigonometric polynomials. We proved it for polynomials in $\xi$, so it remains to prove the formula for polynomials in $\overline \xi$. 

We start with the representation
\begin{align}
\label{M_theta-02}
\cM^*_{\theta_\gamma}  & = M_{\overline z} - M_{\overline z} c^\gamma (c^\gamma)^* - \overline{\theta(0)} c_1^\gamma (c^\gamma)^* 
\\
\notag 
       & = M_{\overline z} + (\overline \gamma c_1^\gamma - M_{\overline z} c^\gamma ) (c^\gamma)^* .
\end{align}
%%
%here we keep in mind that $\theta=\theta_\gamma$, $c = c^\gamma$, $c_1=c_1^\gamma$, 
%but we skip the sub/superscript $\gamma$ to simplify the notation. 
Note that since generally $c_2^\gamma\notin \cK_{\theta_\gamma}$ we cannot get the formula for $\cM_{\theta_\gamma}^*$  by taking the adjoint in \eqref{M_theta-01}, and, as one can see, \eqref{M_theta-02} is not the formal adjoint of \eqref{M_theta-01}. Instead, to get \eqref{M_theta-02} we essentially repeat the reasoning we used to get \eqref{M_theta-01}. 

Namely, for $f\in \cK_{\theta} \cap c^\perp$ we have $ \cM_\theta^* f = M_{\overline z} f$, so the first line of \eqref{M_theta-02} holds for $f\in \cK_{\theta_\gamma}\cap (c^\gamma)^\perp$. 

Since $\cM_{\theta}^* c =\overline{ \theta(0)}c_1$ and $M_{\overline z} c - M_{\overline z} c c^* c =0$, the first line of \eqref{M_theta-02} holds for $f=(c^\gamma)^*$ as well. The second line of  \eqref{M_theta-02} follows from the fact that $\theta_\gamma(0) =-\gamma$.

Substituting \eqref{M_theta-02} into the identity $\Phi_\gamma^* U_\gamma^* = \cM_\theta^* \Phi_\gamma^*$ we get %(skipping $\gamma$ in $\Phi^*_\gamma$)
\begin{align}
\label{Surprise}
\Phi^*_\gamma U_1^* + (\overline\gamma -1) c_1^\gamma b^* = M_{\overline z} \Phi^*_\gamma + (\overline \gamma c_1^\gamma - M_{\overline z} c^\gamma) b^*, 
\end{align}
so
\begin{align}
\label{AdjCommRel}
\Phi^*_\gamma U_1^* & = M_{\overline z} \Phi^*_\gamma + (c_1^\gamma - M_{\overline z} c^\gamma) b^* \\ \notag
             & = M_{\overline z} \Phi^*_\gamma - M_{\overline z} (c^\gamma - c_2^\gamma) b^*, 
\end{align}
where, recall, $c_2^\gamma(z) = z c_1^\gamma(z)$.  Right multiplying \eqref{AdjCommRel} by $U_1^*$ and using  \eqref{AdjCommRel} for $\Phi^*_\gamma U_1^*$ in the right side we get
\[
\Phi^*_\gamma (U_1^*)^2 = (M_{\overline z})^2 \Phi^*_\gamma - (M_{\overline z})^2 (c^\gamma-c_2^\gamma) b^* - M_{\overline z} (c^\gamma - c_2^\gamma) (U_1b)^*  .
\]
Right multiplying the obtained identity by $U_1^*$ and repeating this procedure we get by induction that for all $n\ge 1$
\[
\Phi^*_\gamma (U_1^*)^n = (M_{\overline z})^n \Phi^*_\gamma  - \sum_{k=1}^n (M_{\overline z})^k (c^\gamma-c_2^\gamma) (U_1^{n-k} b)^*
\]
We act as before and apply this formula to the vector $b_1\equiv 1 \in L^2(\mu)$. Then using the identity
\[
-\sum_{k=1}^n (\overline z)^k (\overline \xi)^{n-k} = \frac{ (\overline \xi)^n - (\overline z)^n }{1-\overline \xi z }\,,
\]
we see that the representation formula \eqref{f-reprB} holds for the antianalytic monomials $f(\xi) = (\overline \xi)^n$, $n\in\Z_+$. The representation formula for the analytic monomials is already proved, so by linearity \eqref{f-reprB} holds for all trigonometric polynomials. 

To extend this formula to $C^1(\T)$ we use standard approximation reasoning, similar to one used in \cite{mypaper}. For $f\in C^1(\T)$ take a sequence of trigonometric polynomials $\{p_k\}$ such that $p_k\to f$,  $p_k'\to f'$ uniformly on $\T$. Then $p_k\to f$ in $L^2(\mu)$. 
Note also that $|p_k'|\le C$ (with constant $C<\infty$ independent of $k$).  
We substitute the polynomials  $p_k$ into \eqref{f-reprB} and take the limit as $k\to \infty$.  
Since $\Phi^*_\gamma$ is a bounded operator, the convergence $p_k\to f$ in $L^2(\mu)$ implies  that  $\Phi^{\ast}_\gamma p_k\to \Phi^{\ast}_\gamma f$ in $\K_{\theta_\gamma}$.

To investigate the convergence in the right hand side, consider the functions $f_k:=f-p_k$. We know that $f_k\rightrightarrows 0$, $f_k'\rightrightarrows 0$. 
By the intermediate valued theorem we have
\begin{align*}
|f_k(\xi)- f_k(z)| \le \|f_k'\|_\infty |I_{\xi,z}| \qquad\forall \xi,z\in\T,
\end{align*}
where $I_{\xi, z}\subset \T$ is the shortest arc connecting $\xi $ and $z$. Since $|I_{\xi,z}|\le\frac\pi2 |\xi-z|$ we conclude that 
\[
\left| \frac{f_k(\xi)-f_k(z)}{1-\overline \xi z} \right| \le \frac\pi2 \|f_k'\|_\infty \to 0 \qquad \text{as } k\to \infty, 
\]
so
\begin{align}
\label{IntUnigConv-01}
\int\frac{p_k(\xi)-p_k(z)}{1-\overline \xi z}\, d\mu(\xi)
\,\,\, \rightrightarrows \,\,\,
\int\frac{f(\xi)-f(z)}{1-\overline\xi z}\, d\mu(\xi)
\qquad%\text{for all }
z\in \T.
\end{align}
Recall that our model is in the ambient space $L^2(W) = L^2(\C^2, W)$ (consisting of functions on the circle $\T$ with values in $\C^2$) with a matrix weight $W$. Clearly for a bounded scalar-valued function $\f$ we have in this space that 
\[
\|\f f\|\ci{L^2(W)} \le \|\f\|_\infty \|f\|\ci{L^2(W)} .
\]  

Note that $A_\gamma, \ B_\gamma\in L^2(W)$. 
So, substituting $p_k$ in \eqref{f-reprB} and using the uniform convergence in \eqref{IntUnigConv-01} and the uniform convergence $p_k\rightrightarrows f$ we get  that the right side converges  (in $L^2(W)$ norm) as $k\to \infty$ to the right side of  \eqref{f-reprB} for $f\in C^1$.  The convergence $\Phi_\gamma^* p_k \to \Phi_\gamma^* f$ (in $\cM_{\theta_\gamma}$)  in the left side was already proved before. 
\end{proof}

%The latter proof contains two unexpected coincidences:
%The two rank one perturbations (second summand on the left and right hand sides) of \eqref{Surprise} yield a rank one perturbation.
%And the formulas for $\Phi_\gamma^* \xi^n$ and $\Phi_\gamma^* \bar\xi^n$ can be combined.
%

\section{Singular integral operators and a representation formula for 
\texorpdfstring{$\Phi_\gamma^*$}{Phi<sub>gamma<sup>*} 
adapted to Sz.-Nagy--%
%Foia\c{s}
\texorpdfstring{Foia\c{s}}{Foias} 
transcription}
\label{s:AltRepr}

In the Sz.-Nagy--Foia\c{s} transcription the first entry $g_1$ of the vector $\left(\begin{array}{c} g_1\\ g_2 \end{array}\right) \in\cK_\theta$ belongs to the Hardy space $H^2$. The space $H^2$ can be described as the space of functions analytic in the unit disc $\D$. In this section we give a representation formula for $\Phi_\gamma^*$ in the Sz.-Nagy--Foia\c{s} transcription, which is adapted to the $H^2$ theory. 

We will need some facts about singular integral operators and their regularizations obtained by the authors in 
\cite{mypaper, Regularizations2013}. 

\subsection{Singular integral operators and regularizations}
\label{s:SIO-reg}
Let $\mu$ and $\nu$ be Radon measures on a Haussdorff space $\cX$, and let $K(\fdot, \fdot)$ be a \emph{singular kernel} on $\cX\times\cX$, meaning that $K$ is locally bounded off the ``diagonal'' $\{ (s,t)\in\cX\times\cX: s=t\}$ but can blow up approaching the diagonal.  

Following the standard terminology we say that a bounded operator $T:L^2(\mu)\to L^2(\nu)$ is a \emph{singular integral operator with kernel} $K(s,t)$ if for any bounded functions $f$ and $g$ with disjoint compact supports
\begin{align}
\label{SIO-ScalarPairing}
(Tf, g)\ci{L^2(\nu)} = \int_{\cX\times \cX} K(s,t) f(t) \overline{g(s)} d\mu(t) d\nu(s) . 
\end{align}
One can generalize this notion to the case when our $L^2$ spaces are vector-valued spaces $L^2(\mu, E_1)$ and $L^2(\nu,E_2)$  with values in separable Hilbert spaces $E_1$ and $E_2$ respectively. The values $K(s,t)$ in this case are  bounded linear operators from $E_1$ to $E_2$, and for bounded functions $f$ and $g$ with disjoint compact supports  
\begin{align}
\label{SIO-VectorPairing}
(Tf, g)\ci{L^2(\nu)} = \int_{\cX\times \cX} \Bigl( K(s,t) f(t) , {g(s)} \Bigr)_{E_2} d\mu(t) d\nu(s) . 
\end{align}

Let us interpret the representation formula \eqref{f-reprB} in the Sz.-Nagy--Foia\c{s} transcription as a singular integral operator.

\begin{prop}
\label{p:KernelPhi*}
The operator $\Phi^*_\gamma:L^2(\mu)\to L^2(\C^2)$ in the Sz.-Nagy--Foia\c{s} transcription is a singular integral operator with kernel $\wt K$, 
\[
\wt K(z, \xi) = B_\gamma(z) \frac{1}{1-\overline\xi z}   .
\]
\end{prop}
Note that the kernel in this general case is a $2\times 1$ matrix-function; if the measure $\mu$ is purely singular, then the ``second floor'' of the model collapses and the kernel is scalar. 

\begin{proof}[Proof of Proposition \ref{p:KernelPhi*}]
The identity \eqref{SIO-VectorPairing} for the operator $\Phi_\gamma^*$ and the kernel $\wt K$ clearly holds for $f, g\in C^1(\T)$ with disjoint compact supports.  Let now $f$ and $g$ be bounded functions  with disjoint compact supports. 
Then by taking convolutions of $f$ and $g$ with appropriate smooth mollifiers we get $f_n, g_n\in C^1(\T)$, $f_n \to f $ in $L^2(\mu)$, $g_n\to g$ in $L^2$ as $n\to \infty$ and 
\[
\dist(\supp f_n , \supp g_n ) \ge \delta >0 \qquad \forall n. 
\]
Taking limit as $n\to\infty$ we immediately  get from \eqref{SIO-VectorPairing} for $f_n$ and $g_n$ the corresponding identity for $f$ and $g$.  
\end{proof}

\begin{defn}
\label{d:RestrBdd}
Following \cite{Regularizations2013} we say that a singular kernel $K$ is \emph{restrictedly bounded} (more precisely $L^2(\mu)\to L^2(\nu)$ restrictedly bounded) if 
\[
(Tf, g)\ci{L^2(\nu)} \le C \|f\|\ci{L^2(\mu)}\|g\|\ci{L^2(\nu)}.
\]
The best constant $C$ in the estimate is called the \emph{restricted norm} and denoted by $\|K\|\ti{restr}$. 
\end{defn}

\begin{cor}
The kernel $\wt K$ from Proposition \ref{p:KernelPhi*} is ($L^2(\mu)\to L^2$) restrictedly bounded with the restricted norm at most $1$. Equivalently, the kernel $K$, 
\begin{align*}
K(z, \xi) = \frac{1}{1- \overline\xi z}
\end{align*}
is restrictedly bounded from $L^2(\mu)$ to $L^2(v)$, $v(z) = |B_\gamma(z) |^2$, $z\in \T$ with the restricted norm at most $1$. 
\end{cor}
\begin{proof}
The first statement is trivial, because the kernel of a bounded singular operator is restrictedly bounded with the restricted norm being at most the norm of the operator. 

So, we know that 
\begin{align}
\label{K-pairing-01}
\left|
\int_\T \bigl( \wt K(z, \xi) f(\xi) , g(z) \bigr)\ci{\C^2} d\mu(\xi) dm(z) 
\right|
\le \|f\|\ci{L^2(\mu)} \|g\|\ci{L^2(\C^2)}. 
\end{align}
for all bounded $f$ and $g$ functions with disjoint compact supports. Nothing changes if we consider only the functions $g$ which are $0$ whenever $B(z)=0$,  and which are of form $g= (B/|B|)\f$, where $\f\in L^2$ (scalar-valued), $\|\f\|_2 = \|g\|\ci{L^2(\C^2)}$. Denoting $\psi =\f/|B|$ and noticing that
\[
\|\psi\|\ci{L^2(v)} = \|\f\|\ci{L^2} =\|g\|\ci{L^2(\C)}
\]
we can using $g:=B\psi$ rewrite \eqref{K-pairing-01} as 
\[
\left|
\int_\T  K(z, \xi) f(\xi) \overline{\psi(z)}  v(z) d\mu(\xi) dm(z) 
\right|
\le \|f\|\ci{L^2(\mu)} \|\psi\|\ci{L^2}, 
\]
which is exactly the restricted boundedness of $K$. 
\end{proof}

\begin{theo}
\label{t:UniformBds-K_r}
Let as above $v(z) = | B_\gamma(z) |^2$, $z\in \T$. 
Then for $r\in [0,\infty ) \setminus \{1\}$  the operators $T_r:L^2(\mu) \to L^2(v)$ with kernels $K_r$, 
\[
K_r(z, \xi) =  \frac{1}{1- r \overline\xi z} 
\]
are uniformly (in $r$) bounded. 
\end{theo}
\begin{proof}
First notice, that multiplication of $K$ by $(\overline \xi z)^n$, $n\in \Z$,  does not increase its ($L^2(\mu)$ to $L^2(v)$) restricted norm  (because $(\overline \xi z)^n$ is a product of unimodular functions of $\xi$ and of $z$). 

Then for $0\le r <1$ multiplying $K$ by
\[
1 + \sum_{n=1}^\infty (r^n- r^{n-1}) (\overline \xi z)^n = \frac{1-\overline\xi z}{1-r\overline \xi z}
\]
we at most double the restricted norm (because $1+\sum_{n=1}^\infty |r^n-r^{n-1}| = 2$). 
So, the kernel
\[
K(z, \xi) \cdot \frac{1-\overline\xi z}{1-r\overline \xi z} = \frac{1}{1-r\overline \xi z}
\]
has (for $0\le r<1$) the restricted norm at most $2$. 

If $r>1$ we can write
\begin{align}
\label{SchMult-01}
\frac{1-\overline\xi z}{1-r\overline \xi z} = 1-\sum_{n=1}^\infty(r^{-n} - r^{-(n+1)} ) (\overline \xi z)^{-n}.
\end{align}
Noticing that $1+\sum_{n=1}^\infty |r^{-n} - r^{-(n+1)}| = 1 + r^{-1}\le 2$ we see that in the case $r>1$ 
multiplying $K$ by the function \eqref{SchMult-01} we no more than double the restricted norm. 

Therefore, the restricted norm of the kernels $K_r$, $K_r(z, \xi) =1/(1-r\overline\xi z)$ for $r\in \R_+\setminus\{1\}$ is at most $2$. 

The kernels $K_r$ are bounded, so the corresponding integral operators $T_r$ are well defined. 
Since $K_r\in L^2(\mu\times vdm)$, the operators $T_r$ are Hilbert--Schmidt, and, in particular, compact. 

To complete the proof we use the following result from \cite{Regularizations2013}, see Theorem 3.4 there:
\begin{theo}
\label{t:restr-bd1}
Let $\mu$ and $\nu$ be Radon measures in $\R^N$ without common atoms. Assume that a kernel $K\in L^2\ti{loc} (\mu\times \nu)$ is $L^p$ restrictedly bounded, with the restricted norm $C$. Then the integral operator with $T$ kernel $K$ is a bounded operator $L^p(\mu)\to L^p(\nu)$ with the norm at most $2C$.
\end{theo}

This theorem implies immediately that the operators $T_r:L^2(\mu)\to L^2(v)$ with kernel $K_r$, $r\ne 1$ have the norm at most $4$ (uniformly in $r$). 
\end{proof}

Recall that the Cauchy type operator $R$ was defined  by 
\[
R\nu(z) = \int_{\T} \frac{d\nu(\xi)}{1-\overline{\xi} z }  \,, \qquad z\in \C\setminus\T; 
\]
here   $\nu$ is a (complex) measure of finite variation on $\T$. It is a well-known fact that $K\nu(z)$ (restricted to the disc $\D$ or its exterior) cf \cite{NONTAN}, \cite{poltsara2006}  has non-tangential boundary values a.e.~(with respect to Lebesgue measure) on $\T$. 

For $f\in L^1(\mu)$ let $T_+ f$ and $T_-f$ be the non-tangential boundary values of $(Rf\mu)(z)$ as $z$ approaches $\T$ from inside and from outside of the disc $\D$ respectively.

\begin{prop}
\label{p:wot}
The operators $T_{\pm}:L^2(\mu)\to L^2(v)$ are  bounded and moreover
\[
T_{\pm} = \textup{w.o.t.-}\lim_{r\to 1^\mp} T_r
\]
where \textup{w.o.t.}~means the weak operator topology in the space $B(L^2(\mu), L^2(v))$. 
\end{prop}

\begin{proof}
Let 
$\displaystyle
T_+\ne \textup{w.o.t.-}\lim_{r\to 1-} T_r
$. 
Then for some $f\in L^2(\mu)$ we can find a sequence $r_k\nearrow 1$ such that $T_{r_k} f$ does not converge to $T_+f$ weakly in $L^2(v)$  (the family $T_r f$ is bounded, the weak topology on a bounded set of a Hilbert space is determined by countably many seminorms, so it is metrizable and we can use the definition of limits in terms of sequences). 

By taking a subsequence, if necessary, we get that $T_{r_k} f$ converges weakly to some $g \in L^2(v)$. But we know that $T_{r_k} f \to T_+ f $ a.e., so (see Lemma 3.3 in \cite{mypaper}) $f=g$ and we arrived at a contradiction. 
\end{proof}

\subsection{A representation formula in the
%adapted to the 
%\texorpdfstring{$H^2$}{H<sup>2} theory
Sz.-Nagy--%
%Foia\c{s}
\texorpdfstring{Foia\c{s}}{Foias} 
transcription}
\begin{theo}
\label{t:AltRepr}
The operator $\Phi_\gamma^*$ can be represented in the Sz.-Nagy--Foia\c{s} transcription as
\begin{align}
\label{AltRepr}
(1-|\gamma|^2)^{1/2}\Phi_\gamma^* f & = 
\left( \begin{array}{c} 0 \\  (\overline\gamma - (\overline\gamma - 1) T_+ \1 )\Delta_\gamma \end{array}
\right) f 
+
\left( \begin{array}{c} (1 + \overline\gamma \theta_\gamma) / T_+\1 \\  (\overline\gamma - 1) \Delta_\gamma \end{array}
\right) T_+ f 
\\
\notag
 & = 
\left( \begin{array}{c} 0 \\   \frac{1-\overline\gamma\theta_0}{|1-\overline\gamma\theta_0|}   T_+ \1 \cdot\Delta_0 \end{array}
\right) f 
+
\left( \begin{array}{c} \frac{1-|\gamma|^2}{1 - \overline\gamma \theta_0} \cdot\frac1{ T_+\1} \\  (\overline\gamma - 1) \frac{(1-|\gamma|^2)^{1/2}}{|1-\overline\gamma\theta_0|} \Delta_0 \end{array}
\right) T_+ f 
\end{align}
for $f\in L^2(\mu)$.
\end{theo}

\begin{rem}
\label{T_+1-theta_0}
It is not hard to see and will be shown in the proof of Theorem \ref{t:AltRepr} that the function $1/T_+\1$ is given by the boundary values of the function $1-\theta_0(z) = 1/ R\mu(z)$, $z\in\D$.  The function $T_+f$ is given by the boundary values of the function $Rf\mu(z)$, $z\in\D$. 

The function $R\mu(z)$ belongs to the Nevanlinna class, and so does the function 
\begin{align}
\label{UpFloor}
 (1+\overline \gamma\theta_\gamma(z)) \frac{Rf\mu(z)}{R\mu(z)} =  (1+\overline \gamma\theta_\gamma(z)) (1-\theta_0(z)) Rf\mu(z), \qquad z\in \D. 
\end{align}
The (non-tangential) boundary values of the function \eqref{UpFloor} give us the top entry in the right side of \eqref{AltRepr}, which as we know belongs to $H^2$ (treated as a subspace of $L^2(\T)$). The uniqueness theorem for functions of Nevanlinna class (such functions are determined by their boundary values) implies that the harmonic extension of this $H^2$ function (which is analytic in $\D$) is given by \eqref{UpFloor}. 

Thus, the function from \eqref{UpFloor} belongs to $H^2$. Since the reciprocal of the function $1 + \overline\gamma \theta_\gamma$ is in $H^\infty$, we can conclude that $Rf\mu/R\mu\bigm|_{\textstyle \D} \in H^2$.
\end{rem}

\begin{rem}
If the measure $\mu$ is purely singular (equivalently, $\theta$ is an inner function), the formula for $\Phi^*$ was obtained (modulo the existence of boundary values $\mu$-a.e.) by D.~Clark \cite{Clark}. 
%The existence of boundary values was proved by A.~Poltoratskii \cite{NONTAN}. 
Note that in \cite{Clark} the measure $\mu$ was not supposed to be a probability measure if $\gamma\ne 0$, so some computations are  required to see that the formulas in \cite{Clark} for $\gamma\ne0$ give the same result us in this paper. 
\end{rem}

\begin{proof}[Proof of Theorem \ref{t:AltRepr}]
First, let us notice that for $f\in C^1$
\begin{align}
\label{T_pm-01}
(T_r f) (z) - f (z) (T_r 1)(z) = \int_\T \frac{f(\xi) - f(z) }{1-r\overline\xi z} d\mu(\xi) 
\quad \rightrightarrows \quad\int_\T \frac{f(\xi) - f(z) }{1- \overline\xi z} d\mu(\xi) 
\end{align}
uniformly on $\T$ as $r\to 1$. 

On the other hand, by Proposition \ref{p:wot}, $T_r f \to T_+ f$ weakly in $L^2(v)$ as $r\to 1-$, so taking the weak limit as $r\to1-$ in the left hand side of \eqref{T_pm-01} we get that for $f\in C^1$
\[
(T_+ f) (z) - f (z) (T_+ 1)(z)  = \int_\T \frac{f(\xi) - f(z) }{1- \overline\xi z} d\mu(\xi), \qquad z\in\T  . 
\]
Then for all $f\in L^2(\mu)$
\[
\Phi_\gamma^* f (z) = A_\gamma(z) f(z) + B_\gamma(z) \bigl( (T_+f)(z) - f(z) (T_+1)(z) \bigr). 
\]
Note, that while we initially proved this formula only for $f\in C^1(\T)$, we can extend it by continuity from the dense set $C^1(\T)$ using the fact that the operator $f\mapsto B_\gamma T_+f$ is bounded.

The remainder of the proof is given by rather straightforward calculations. 

Namely, denoting 
\[
(1-|\gamma|^2)^{1/2} \Phi_\gamma^* f =: 
\left( \begin{array}{c} g_1 \\ g_2\end{array}\right) 
\]
and recalling the formulas \eqref{A(z)} and \eqref{B(z)} for $A_\gamma$ and $B_\gamma$ we get 
\begin{align*}
g_2 & = \overline\gamma \Delta  f   + (\overline\gamma -1) \Delta \cdot (T_+f - f T_+1 ) \\
& = \bigl( \overline\gamma - (\overline\gamma-1) T_+1 \bigr) f \Delta + (\overline\gamma-1) \Delta T_+f, 
\end{align*}
%$
which gives us the lower entry in the right side of \eqref{AltRepr}. 

The computation of $g_1$ is just a bit more complicated. 
Namely, we can write recalling the definition of $A$ and $B$
\begin{align}
\label{g1-01}
g_1 = (1-\overline\gamma\theta_\gamma) f + \bigl( 1+(\overline\gamma -1)\theta_\gamma - \gamma) (T_+f - f T_+1)  .
\end{align}
Using \eqref{theta_0-theta_gamma} we can write
\begin{align*}
1 + (\overline\gamma -1) \theta_\gamma - \gamma = 1+ \overline \gamma \theta_\gamma - \theta_\gamma - \gamma
&= ( 1+ \overline \gamma \theta_\gamma ) \left( 1 - \frac{\theta_\gamma + \gamma}{ 1+ \overline \gamma \theta_\gamma } \right) 
%\\
%& 
= ( 1+ \overline \gamma \theta_\gamma ) (1-\theta_0). 
\end{align*}
Now recall that in $\D$
\[
\theta_0 = \frac{R_1\mu}{1+R_1\mu} = \frac{R_1\mu}{R\mu}, 
\]
so
\begin{align}
\label{1-theta_0}
1-\theta_0 = \frac{1}{1+R_1\mu} = \frac{1}{R\mu}. 
\end{align}
Taking the non-tangential boundary values we get that on the unit circle $\T$
\begin{align}
\label{1-theta_0-01}
1-\theta_0 = 1/T_+\1. 
\end{align}
Substituting all this in \eqref{g1-01} we get 
\begin{align*}
g_1 & = (1-\overline\theta_\gamma) f - (1+\overline\gamma\theta_\gamma)(1-\theta_0) (T_+\1) f +
(1+\overline\gamma\theta_\gamma)(1-\theta_0) T_+f 
\\
& = (1-\overline\theta_\gamma) f  - (1-\overline\theta_\gamma) f  + (1+\overline\gamma\theta_\gamma)(T_+\1)^{-1} T_+f 
\\
&= (1+\overline\gamma\theta_\gamma)(T_+\1)^{-1} T_+f , 
\end{align*}
which gives us the first line of \eqref{AltRepr}. Using the identity $\theta_\gamma =(\theta_0-\gamma)/(1-\overline\gamma\theta_0)$ we get that
\[
1+\overline\gamma\theta_\gamma = \frac{1-|\gamma|^2}{1 - \overline\gamma \theta_0}, 
\]
which gives us the second line in \eqref{AltRepr}.
\end{proof}

\section{A representation for 
%$\Phi_\gamma^*$
\texorpdfstring{$\Phi_\gamma^*$}{Phi<sub>gamma<sup>*} 
in the de Branges--Rovnyak transcription and a formula for 
%$\Phi_\gamma$ 
\texorpdfstring{$\Phi_\gamma$}{Phi<sub>gamma}
}
\label{s:Phi^*_gamma-deBranges}

\subsection{%A representation of 
%%$\Phi_\gamma^*$
%\texorpdfstring{$\Phi_\gamma^*$}{Phi<sub>gamma<sup>*} 
%in 
Preliminaries about the de Branges--Rovnyak transcription}

%We get an alternative representation of the operator $\Phi_\gamma^*$ which uses the fact that a 
%function $g\in\cK_\theta$ can be defined be two functions; one  in $H^2$ (analytic in the unit disc) 
%and the other in $H^2_- := L^2 (\T) \ominus H^2$ (analytic in the exterior of the unit disc). 

It is easy to see from the definition that in the Sz.-Nagy--Foia\c{s} transcription a function 
\[
g=\left(\begin{array}{c} g_1\\ g_2\end{array}\right) \in \left(\begin{array}{c} H^2 \\ \clos\Delta L^2 \end{array} \right)
\] 
is in $\cK_\theta$ if and only if 
\begin{align}
\label{deBrangesRepr-01}
g_- := \overline \theta g_1 + \Delta g_2 \in H^2_- := L^2 (\T) \ominus H^2 .
\end{align}
Note, that knowing $g_1$ and $g_-$ one can restore $g_2$ on $\T$:
\[
g_2\Delta = g_-  -   g_1\overline \theta. 
\]

The   equality \eqref{deBrangesRepr-01} means that the pair $g_+ = g_1$ and $g_-$ belongs to the de Branges--Rovnyak space, see \eqref{deBrangesRepr}. 
It is also not hard to check that the norm of the pair $(g_1, g_-)$ in the Branges--Rovnyak space (i.e.~in the weighted space $L^2(W)$, 
$W=W_\theta^{[-1]}$, see Section~\ref{s:deBrangesRepr}) coincides with the norm of the pair $(g_1, g_2)$ in the Sz.-Nagy--Foia\c{s} 
space (i.e.~in non-weighted $L^2$). Indeed, we have
\begin{align*}
\left(\begin{array}{c} g_1 \\ g_- \end{array}\right) = 
\left(\begin{array}{cc} 1 & 0 \\ \overline\theta & \Delta \end{array}\right)
\left(\begin{array}{c} g_1 \\ g_2 \end{array}\right)\,.
\end{align*}
Let $B$ be a ``Borel support'' of $\Delta$, i.e.~the set where one of the representative from the equivalence class of $\Delta$ is different from $0$. 
A direct computation shows that for 
\[
W_\theta = \left(\begin{array}{cc} 1 & \theta \\ \overline\theta & 1 \end{array}\right)
\]
we have a.e.~on $\T$ 
\begin{align*}
\left(\begin{array}{cc} 1 & \theta \\ 0 & \Delta \end{array}\right)
W_\theta^{[-1]}\left(\begin{array}{cc} 1 & 0 \\ \overline\theta & \Delta \end{array}\right)
=
\left(\begin{array}{cc} 1 & 0 \\ 0 & \1\ci B \end{array}\right) ,  
\end{align*}
which gives the desired equality of the norms.

Note that functions in $H^2_-$ admit analytic continuation to the exterior of the unit disc, so a function in $\cK_\theta$ is determined by the boundary values of two functions $g_1$ and $g_-$ analytic in $\D$ and $\ext \D$ respectively.

\begin{rem}
\label{r:dBR_model}

When $\theta$ is an extreme point of the unit ball in $H^\infty$ the de Branges--Rovnyak transcription reduces to studying only the function $g_1$ at the expense of having the more complicated range norm. Let us recall that a function $\theta\in H^\infty$, $\|\theta\|_\infty\le1$ is an extreme point of the unit ball in $H^\infty$ if and only if 
\[
\int_\T \ln (1-|\theta|^2) |dz| =-\infty. 
\]
Since (see Remark \ref{l:theta_0-theta_gamma}) we have that  $\Delta_\gamma=(1-|\theta_\gamma|^2)^{1/2}$ satisfy 
\[
%\theta_\gamma = \frac{\theta_0 -\gamma}{1-\overline\gamma\theta_0}
\Delta_\gamma = \frac{(1- |\gamma|^2)^{1/2}}{|1-\overline\gamma\theta_0|} \Delta_0,
\]
a function $\theta_\gamma$, $\gamma\ne0$ is an extreme point if and only if $\theta_0$ is. 

It is not hard to compute using Lemma \ref{l:CharFunct-U_0} that 
\begin{align}
\label{Delta--w}
1-|\theta_0|^2 =: \Delta_0^2 = |1-\theta_0|^2 w, 
\end{align}
where $w$ is the density of the absolutely continuous part of $\mu$. Since $1-\theta_0\in H^\infty$ we have that $\int_\T |1-\theta_0|^2 |dz| $ is finite, so $\theta_0$ (and thus all $\theta_\gamma$, $|\gamma|<1$) is an extreme point if and only if $\int_\T \ln w |dz|>-\infty$. 

If $\theta$ is an extreme point in the unit ball of $H^\infty$, then the function $g_-$ is defined uniquely by the function $g_1$: indeed, if 
\begin{align*}
g_- = \overline\theta g_1 + \Delta g_2 \in H^2_- \qquad \text{and} \qquad  \wt g_- = \overline\theta g_1 + \Delta \wt g_2 \in H^2_-
\end{align*}
then
\[
g_- - \wt g_- = \Delta (g_2- \wt g_2).
\]
Assume that $g_- \ne \wt g_-$. Since $g_- - \wt g_- \in H^2_-$, we have $\int \ln_\T |g_- - \wt g_-| |dz|>-\infty$. We know that $g_2-\wt g_2\in L^2$, so $\int_\T \ln |g_2-\wt g_2| |dz| <\infty$. But $\int_\T \ln \Delta |dz|=-\infty$ (because $\theta$ is an extreme point), so 
\[
-\infty< \int_\T \ln |g_-  - \wt g_- | \,|dz| = \int_\T \ln \Delta \,|dz| + \int_\T \ln |g_2-\wt g_2|\,|dz| = -\infty, 
\] 
which is a contradiction. Thus $g_- = \wt g_-$ and $g_2 = \wt g_2$.

In fact even more is true: if $\theta$ is an extreme point, the map
\begin{align}
\label{i^*}
(g_1, g_2) \mapsto g_1 
\end{align}
maps unitarily the de Branges--Rovnyak space $\cD(\theta)$ onto $\cH(\theta)$, see Remark \ref{r:dBr-R_orig} for the definitions. This is true, for example, because according to  \cite[Lemma 12.1]{Nik-Vas_TwoModels_1985} $\theta\in H^\infty$, $\|\theta\|_\infty\le1$ the map \eqref{i^*} is a coisometry (adjoint of an isometry). And, as we discussed just before, if $\theta$ is an extreme point, the map \eqref{i^*} has trivial kernel. 

Note that the transformation \eqref{i^*} intertwines the backward shifts, so it gives us a canonical isomorphism between $\cD(\theta)$ and $\cH(\theta)$.

Using the described in Remark \ref{r:dBr-R_orig} connection between the space $\cD(\theta)$ and $\cK_\theta$ in the de~Branges--Rovnyak transcription, we can see that the above map maps unitarily the model space $\cK_\theta$ (in both de Branges--Rovnyak and Sz.-Nagy--Foia\c{s} transcriptions) onto $\cH(\theta)$.

%let us explain why we can ignore $g_-$ using standard facts found in \cite{SAR}. 
%The function $\theta$ is extreme if and only if $\int_\T \log(1-|\te|^2) = -\infty$. 
%Or equivalently, we have $L^2(\mu) = H^2(\mu)$, since the absolutely continuous part 
%$w$ of $\mu$ is given by $w = \frac{1-|\te|^2}{|1-\te|^2}$ and by Szeg\"o's theorem. 
%In other words, when $\te$ is extreme, it suffices to consider $H^2(\mu)$. It remains 
%to observe (see III-7 of \cite{SAR}) that the normalized Cauchy transform isometrically 
%maps $H^2(\mu)$ onto the analytic component $g_+$ of the de Branges--Rovnyak space.

\end{rem}

\subsection{Computing \texorpdfstring{$g_-$}{g<sub>-} for  the de Branges--Rovnyak transcription}
As it was shown above in Section \ref{s:AltRepr} for the function $g= \Phi_\gamma^* f \in \cK_{\theta_\gamma}$ the entry $g_1$ is determined in the unit disc by the \emph{normalized Cauchy transform} $Rf\mu/R\mu$, 
\begin{align}
\label{NCT-01}
g_1(z) = (1-|\gamma|^2)^{-1/2} (1+\overline\gamma \theta_\gamma)\, \frac{Rf\mu}{R\mu} 
& = \frac{(1-|\gamma|^2)^{1/2}}{1-\overline{\gamma}\theta_0}\, \frac{Rf\mu}{R\mu} 
\\
\notag
& =  \frac{(1-|\gamma|^2)^{1/2}(1-\theta_0)}{1-\overline{\gamma}\theta_0}\, Rf\mu \,. 
\end{align}
In order to complete the representation of $\Phi^*_\gamma$ in the de Branges--Rovnyak transcription, it remains to find a similar representation for $g_-^\gamma$.

\begin{theo}
\label{t:g_-}
Let $\mu$ be not the Lebesgue measure. 
Then the function $g_-=g^\gamma_-$ is given by 
\begin{align}
\label{g_-}
g^\gamma_- = (1-|\gamma|^2)^{-1/2} \left( \overline \theta_\gamma +\overline\gamma \right) \frac{T_-f}{T_-\1} 
& = \frac{(1-|\gamma|^2)^{1/2} \overline\theta_0 }{1-\gamma \overline\theta_0}  
\cdot 
\frac{T_-f}{T_-\1}  \\
\notag
& =  \frac{(1-|\gamma|^2)^{1/2} (1-\overline\theta_0) }{1-\gamma \overline\theta_0}    T_-f    \,.
\end{align}
\end{theo}

\begin{rem}
\label{r:g_-Alt}
Defining the analytic outside of the disc function $\theta_\gamma^\sharp$ by $\theta^\sharp_\gamma(z) = \overline{\theta_\gamma(1/\overline z)}$, $|z|>1$ we can see that the values of $g_-^\gamma$ on $\T$ are the non-tangential boundary values analytic on $\{z\in\C:|z|>1\}$ function 
\[
(1-|\gamma|^2)^{-1/2} \left( \theta^\sharp_\gamma(z)  + \overline\gamma \right) \frac{(Rf\mu)(z)}{(R\mu)(z)}\,, \qquad |z|>1. 
\]
\end{rem}

\begin{proof}[Proof of Theorem \ref{t:g_-}]
Acting as in the proof of Theorem \ref{t:AltRepr}  but now taking the weak limit as $r\to 1+$ in the left hand side of \eqref{T_pm-01}, we get that  for $f\in C^1(\T)$ 
\[
(T_- f) (z) - f (z) (T_- 1)(z)  = \int_\T \frac{f(\xi) - f(z) }{1- \overline\xi z} d\mu(\xi), \qquad z\in\T  . 
\]
This gives us the following representation for $\Phi_\gamma^* f$, $f\in L^2$, 
\[
\Phi_\gamma^* f (z) = A_\gamma(z) f(z) + B_\gamma(z) \bigl( (T_-f)(z) - f(z) (T_-1)(z) \bigr). 
\]
Again, we proved this formula for $f\in C^1(\T)$, but we can extend it from the dense set to all $L^2(\mu)$ by using the fact that the operators $T_-:L^2(\mu)\to L^2(v)$ and $\Phi_\gamma^* L^2(\mu) \to \cK_\theta$ are bounded. 

To simplify the formulas denote $\wt g_- := (1-|\gamma|^2) g_-$. Recalling the formulas \eqref{A(z)}, \eqref{B(z)} for $A_\gamma$ and $B_\gamma$ we can then write
\begin{align*}
\wt g_- & = (\overline\theta_\gamma +\overline\gamma |\theta_\gamma|^2) f + \overline\gamma \Delta_\gamma^2 f  
+
(\overline\theta_\gamma + (\overline\gamma -1) |\theta_\gamma|^2 - \gamma \overline\theta_\gamma)
(T_-f - fT_-\1) 
\\
& 
\qquad \qquad \qquad \qquad\qquad\qquad\qquad\qquad
+ (\overline\gamma-1) \Delta_\gamma^2 (T_-f - fT_-\1)  \\
& = 
(\overline\theta_\gamma + \overline\gamma) f 
+
\bigl( (1-\gamma)\overline\theta_\gamma + \overline\gamma -1 \bigr) (T_- f - f T_-\1) 
\\ 
&= \Bigl( \overline\theta_\gamma + \overline\gamma - \left[ (1-\gamma)\overline\theta_\gamma + \overline\gamma -1 \right] T_-\1 \Bigr) f +  
\left[ (1-\gamma)\overline\theta_\gamma + \overline\gamma -1 \right] T_- f . 
\end{align*}
Since $\wt g_-\in H^2_-$ we can conclude that the first term in the last line above always equals $0$, so
\begin{align}
\label{brac-01}
(\overline\theta_\gamma + \overline\gamma)/T_-\1 
= 
(1-\gamma)\overline\theta_\gamma + \overline\gamma - 1. 
\end{align}
This conclusion, of course, requires some reasoning, which we will not present here. Instead, we will prove \eqref{brac-01} using direct calculations. 

Assuming for a moment that \eqref{brac-01} is proved, we can immediately see from the calculations above that 
\[
\wt g_- = \Bigl(\overline\theta_\gamma + \overline\gamma\Bigr)\frac{T_-f}{T_-\1}, 
\]
which proves the first equality in the theorem. The second equality is easily obtained using the fact that $\theta_\gamma=(\theta_0-\gamma)/(1-\overline\gamma\theta_0)$. 

To prove \eqref{brac-01}  notice that for $r\in(0,1)$ the kernel of the operator $T_{1/r}$  can be written as 
\[
\frac{1}{1-r^{-1} \overline\xi z} = \frac{-r\xi\,\overline z}{1- r\xi\,\overline z} = 1 - 
\frac{1}{1- r\xi\,\overline z} \,.
\] 
Therefore, using the fact that $\mu(\T)=1$ and recalling (see Remark \ref{T_+1-theta_0}) that $T_+\1= 1/(1-\theta_0)$, we can see that 
\begin{align}
\label{T_-1}
T_-\1 = 1 - \overline{T_+\1} = 1 - \frac{1}{1-\overline\theta_0} = 
\frac{-\overline\theta_0}{1- \overline\theta_0} \,.
\end{align}
Taking this into account we can rewrite \eqref{brac-01} as 
\begin{align}
\label{brac-02}
(\overline\theta_\gamma + \overline\gamma) ( 1 - 1/\overline\theta_0) 
= 
(1-\gamma)\overline\theta_\gamma + \overline\gamma - 1.
\end{align}
Using the identity $\theta_\gamma = (\theta_0 -\gamma)/(1-\overline\gamma\theta_0)$ one can easily compute that both sides of \eqref{brac-02} (and of \eqref{brac-01}) equal
\[
(1-|\gamma|^2)\frac{\overline\theta_0-1}{1-\gamma\overline\theta_0}\,. 
\]
This completes the proof of \eqref{brac-01}, and so of the theorem. 
\end{proof}

\begin{rem}
\label{Sarason}
If $\theta_\gamma$ is an extreme function of the unit ball in $H^\infty$ (which is equivalent to $\int_\T \ln w \,|dz|=-\infty$) the operator $\Phi_\gamma^*$ was described by D.~Sarason \cite{SAR}. His operator acts from $L^2(\mu)$ to the de Branges space $\cH(\theta)$, but as  we said in Remark \ref{r:dBR_model} above, there is a canonical isomorphism from between $\cH(\theta)$ and the de Branges--Rovnyak model space $\cD(\theta)$. 

The formulas obtained in \cite{SAR} coincide with the formula \eqref{NCT-01} for $g_1$ obtained above. This can be seen directly if $\gamma=0$  (equivalently $\theta(0)=0$): in the case $\gamma\ne 0$ some computations are necessary, since in \cite{SAR} a different normalization of the measure $\mu$ is used: unlike this paper, the measure $\mu$ there is not supposed to be a probability measure. 

Note, that it was shown in \cite{SAR} that the formula \eqref{NCT-01} for $g_1$ defines a unitary operator from $H^2(\mu):= \spa\ci{L^2(\mu)}\{z^n: n\in \Z_+\}$ to $\cH(b)$. 

%\eqref{A(z)}
\end{rem}

\subsection{Representation of 
%$\Phi_\gamma$
\texorpdfstring{$\Phi_\gamma$}{Phi<sub>gamma}
}
In this section we give a representation of the Clark operator $\Phi_\gamma$. Any function $f\in L^2(\mu) $ can be decomposed as the sum $f= f\ti s + f\ti a$ of the ``singular'' and ``absolutely continuous'' parts $f\ti s$ and  $f\ti a$. Formally, $f\ti s$ and $f\ti a$ can be defined as Radon--Nikodym derivatives $f\ti s = d(f\mu)\ti s/ d\mu\ti s$, $f\ti a = d(f\mu)\ti a/ d\mu\ti a$. 

Let $w$ denote the weight of the absolutely continuous part of $d\mu$, i.e.~$w = d\mu/dx\in L^1$.

\begin{theo}
\label{t:ReprPhi}
Let $g = \left(\begin{array}{c} g_1\\ g_2 \end{array}\right) \in \cK_{\theta_\gamma}$ (in the Sz.-Nagy--Foia\c{s} transcription) and let $f\in L^2(\mu)$, $f= \Phi_\gamma g$. Then
\begin{enumerate}
\item the non-tangential boundary values of the function 
\[
z\mapsto\frac{1-\overline\gamma}{(1-|\gamma|^2)^{1/2}} g_1(z), \qquad z\in \D
\]
exist and coincide with $f\ti s$ $\mu\ti s$-a.e.~on $\T$. 
 
\item for the ``absolutely continuous'' part $f\ti a$ of $f$ 
\[
(1-|\gamma|^2)^{1/2} w f\ti a  = \frac{1-\overline\gamma\theta_0}{1-\theta_0} g_1 + \frac{1-\gamma\overline\theta_0}{1-\overline\theta_0} g_-
\]
a.e.~on $\T$; here, recall, $g_-:= g_1\overline \theta_\gamma + \Delta_\gamma g_2\in H^2_-$.  
\end{enumerate}
\end{theo}

%\begin{rem*}
%Part \cond1 
%\end{rem*}

To prove statement \cond2 of the theorem we need the following simple and well-known fact. 

\begin{lem}
\label{T_+-T_-}
$T_+ f - T_- f = w f$ a.e.~on $\T$ (with respect to the Lebesgue measure) for all $f\in L^2(\mu)$. 
\end{lem}
\begin{proof}
Consider the operator $T_r - T_{1/r}$, $r\in (0,1)$. Its kernel is given by
\[
\frac{1}{1- r\overline \xi z} - \frac{1}{1- r^{-1}\overline \xi z} = 
\frac{1}{1- r\overline \xi z} + \frac{r\overline z\xi}{1- r\overline  z\xi}=
\frac{1-r^2}{|1- r\overline \xi z|^2} \,.
\]
Therefore for almost all $z\in\T$ 
\[
T_+f(z) - T_-f(z) = \lim_{r\to 1-} \int_\T \frac{1-r^2}{|1- r\overline \xi z|^2}  f(\xi) d\mu(\xi), 
\]
and the conclusion follows from the Fatou Theorem (cf \cite[Ch.~I, Theorem 5.3]{Garnett-BAFbook-2007}). 
\end{proof}

\begin{proof}[Proof of 
Theorem \ref{t:ReprPhi}]
Let us first prove statement \cond2. 
Recall that we have proved before, see Theorems \ref{t:AltRepr} and    \ref{t:g_-},  that
\begin{align}
\label{g_1&g_-}
g_1 = 
\frac{(1-|\gamma|^2)^{1/2}}{1 - \overline\gamma \theta_0} \cdot\frac{T_+f}{ T_+\1}\,, \qquad 
g_- 
= \frac{(1-|\gamma|^2)^{1/2} \overline\theta_0 }{1-\gamma \overline\theta_0}  
\cdot 
\frac{T_-f}{T_-\1}\,.
\end{align}
Using the identities $1/T_+\1 = 1-\theta_0$ and $1/T_-\1 = -\overline\theta_0/(1-\overline\theta_0)$ (see \eqref{1-theta_0} and \eqref{T_-1}) we get
\begin{align}
\label{e-g1}
g_1 = (1-|\gamma|^2)^{1/2}\frac{1-\theta_0}{1-\overline\gamma\theta_0}T_+ f\,, \qquad
g_- = -(1-|\gamma|^2)^{1/2}\frac{1-\overline\theta_0}{1-\gamma\overline\theta_0} T_-f\,,
\end{align}
so
\[
\frac{1-\overline\gamma\theta_0}{1-\theta_0} g_1 + 
\frac{1-\gamma\overline\theta_0}{1-\overline\theta_0} g_-
= (1-|\gamma|^2)^{1/2} (T_+ f - T_- f) = (1-|\gamma|^2)^{1/2} w f;
\]
the last equality here follows from  Lemma \ref{T_+-T_-}. 

To prove \cond1 we recall that according to Poltoratski's Theorem, see \cite[Theorem 2.7]{NONTAN}, 
%see also  \cite{NPPOL})  
the non-tangential boundary values of the normalized Cauchy Transform $Rf\mu/R\mu$ coincide  $\mu\ti s$-a.e.~with $f\ti s$. 
It is well known that $\mu\ti s$-a.e.~on $\T$ the non-tangential boundary values of $R\mu(z)$ are $\infty$, which together with the identity $1/T_+1 = 1-\theta_0$ implies that the non-tangential boundary values of $\theta_0$ equal $1$ wrt $\mu\ti{s}$-a.e.~on $\T$. Substituting the boundary values of $Rf\mu/R\mu$ and of $\theta_0$ into the first equality in \eqref{g_1&g_-} we immediately get statement \cond1. 
\end{proof}
%The representation for $f\ti s = (\Phi_\gamma g)\ti s$ essentially coincides with 

\section{Boundedness of the normalized Cauchy transform and its generalizations}
\label{s:boundNCT}
It was proved in \cite{NONTAN} that the so called \emph{normalized Cauchy transform}, which in our notation is the operator $f\mapsto T_+f /T_+\1$ is bounded from $L^2(\mu)$ to $L^2$. In this paper we proved a slightly stronger result. 

Namely, we proved (see Proposition \ref{p:wot}) that for $v_\gamma = |B_\gamma|^2$ the operator $T_\pm:L^2(\mu)\to L^2(v_\gamma)$ is a bounded operator. 
%Looking at \eqref{AltRepr}

One can see from the proof of Theorem \ref{t:AltRepr} that $B_\gamma$ is the vector coefficient for $T_+ f$ in 
\eqref{AltRepr}. So one can see from \eqref{AltRepr} that all $v_\gamma$ are equivalent in the sense of two sided estimates, so if we are not after sharp constants, we can only consider one value of $\gamma$, say $\gamma=0$; considering other $\gamma$ does not tell us anything new about boundedness of the operator $T_+$. 

Vice versa, for $\gamma=0$ we get using the identity $1/T_+\1 = 1-\theta_0$ (see Remark \ref{T_+1-theta_0})
\begin{align}
\label{v_0}
v_0 (z) = |B_0(z)|^2 = |1-\theta_0(z)|^2 + \Delta_0(z)^2 = 2\re (1-\theta_0(z)) , \qquad z\in\T  . 
\end{align}

The boundedness of the normalized Cauchy transform is equivalent to the boundedness of $T_+:L^2(\mu)\to L^2(v)$, $v=|1-\theta_0|^2$;   generally $v$ can be significantly smaller than $v_0$.  %(example here????)
Indeed, if $\theta_0(z)$ approaches $1$ non-tangentially as $z\to z_0$  ($z, z_0\in\T$), then $\re (1-\theta_0(z) )\asymp |1-\theta_0(z)|$. Therefore, as one can see from \eqref{v_0},  in this case $v(z) \ll v_0(z)$  ($v(z) \le C v_0(z)^2$) as $z\to z_0$. 

For such $\theta_0$ one can take a conformal mapping from the unit disc $\D$ to a Jordan domain $\Omega\subset \D$ touching the point $1$ non-tangentially; we also require that $\theta_0(0)=0$. 

From such $\theta_0$ one can construct the corresponding measure $\mu$ via the standard for the Clark theory recipie:  rewriting the formula from Lemma \ref{l:CharFunct-U_0} as
\begin{align}
\label{theta-to-mu}
R_2\mu = \frac{1+\theta_0}{1-\theta_0}
\end{align}
we can restore the measure $\mu$ as the unique measure satisfying \eqref{theta-to-mu}. This is equivalent to finding the measure $\mu$ for which the Poisson extension to the unit disc equals 
\begin{align}
\label{ReR_2}
\re \frac{1+\theta_0}{1-\theta_0} =\frac{1-|\theta_0|^2}{|1-\theta_0|^2}  . 
\end{align}
Then $R_2\mu$ is the analytic function with real part given by \eqref{ReR_2} and such that $\im R_2\mu(0) =0$, which means that $R_2\mu$ is given by \eqref{theta-to-mu}

\subsection{Exterior Cauchy transform}
Let us now investigate the \emph{exterior} normalized Cauchy transform $f\mapsto T_- f /T_-\1$. One would expect that this operator is bounded $L^2(\mu)\to L^2$, but that is not the case!  
Namely, let us consider a function $\theta_0$ such that in a small neighborhood $E$ of $i\in \T$
\begin{align}
\label{neigh-01}
|1-\theta_0|\ge 1/2 \qquad \text{and} \qquad 1/\theta_0\notin L^2 ,
\end{align}
but the measure $\mu$ does not vanish in a small neighborhood $F$ of $1$. 

This can be easily done by considering an appropriate outer function and multiplying it by $z$ to get $\theta_0(0)=0$. Since $|\theta_0| $ can be an arbitrary function such that $0\le |\theta_0|\le 1$ and $\int_\T \ln |\theta_0(z)|\,|dz|>-\infty$, we can easily construct $\theta_0$ that $1/\theta_0\notin L^2$ in a small neighborhood of the point  $i$. We can also easily get that $\Delta_0$ does not vanish in a neighborhood of $1$, which in light of \eqref{Delta--w} means that $w$ does not vanish there as well. 

Now, we know (see \eqref{T_-1}) that $1/T_-\1 =  \overline\theta_0/(\overline\theta_0-1)$, so 
\[
\frac{T_- f}{T_-\1} = \frac{\overline\theta_0-1}{\overline\theta_0} T_- f. 
\]
Take $f=\1\ci F$, where $F$ is the small neighborhood of $1$ where the measure $\mu$ does not vanish. If the neighborhoods $E$ and $F$ are sufficiently small, we get that $|T_- f|\ge \delta>0$ in $E$, which together with \eqref{neigh-01} gives us that $T_-f/T_-\1 \notin L^2$. 

But what can be said about boundedness of $T_-$? We have proved, see Proposition \ref{p:wot}, that the operator $T_-:L^2(\mu)\to L^2(v_\gamma)$ is bounded. Considering $\gamma=0$ (as we just discussed it is enough to consider only this case) we can obtain a result about boundedness of the exterior normalized Cauchy transform. Namely, this implies that the operator 
\begin{align}
\label{ExtNormCauchy}
f\mapsto \overline \theta_0 \frac{T_-f}{T_-\1}
\end{align}
is a bounded operator from $L^2(\mu)$ to $L^2$. Note that the right side of \eqref{ExtNormCauchy} is given by the boundary values of a function analytic in the exterior of the unit disc, so the ``correct'' exterior normalized Cauchy transform shoud be the operator 
\[
f\mapsto \overline{\theta_0(1/\overline z) }\frac{Rf\mu(z)}{R\mu(z)}, \qquad |z|>1. 
\]

\section{Formula for other Clark measures}
\label{s:OtherClarkMeas}
 
 In this section we consider the Clark operator $\Phi_{\alpha, \gamma} :\cK_{\theta_\gamma} \to L^2(\mu_\alpha)$, where $\mu_\alpha$, $|\alpha|=1$ is the spectral measure (corresponding to the cyclic vector $b$ of the unitary operator $U_\alpha$). The Clark operator $\Phi_{\alpha,\gamma}$ intertwines the model operator $\cM_{\theta_\gamma}$ and the operator $(U_\gamma)_\alpha$ which is the operator $U_\gamma$ in the spectral representation of the operator $U_\alpha$, 
 \[
 \Phi_{\alpha,\gamma} \cM_{\theta_\gamma}  = (U_\gamma)_\alpha \Phi_{\alpha,\gamma} \,
 \]
 (the original operator $\Phi_{\gamma}$ did the same, but with $U_\gamma$ in the spectral representation of $U_1$).

We reduce everything to the case $\alpha=1$ considered before. 
Namely, we can write 
\begin{align}
\notag
U_\gamma = U + (\gamma -1 ) b b_1^*  & = U + (\alpha -1) bb_1^* + (\gamma - \alpha) b b_1^* 
\\
\label{U_alpha_gamma}
& = U_\alpha + (\gamma/\alpha -1) b (\overline\alpha b_1)^* = U_\alpha + (\wt\gamma -1 ) b \wt b_1^*. 
\end{align}
Then we are in the situation treated above, with $U_\alpha$ instead of $U=U_1$ and with $\wt\gamma =\gamma/\alpha$ instead of $\gamma$. 
Note that $\wt b_1 =\overline\alpha b_1$ plays the role of $b_1$ here: indeed, $U_\alpha \wt b_1= b$. 

Now let us see what modifications should be done to the main formulas. 

To keep track of the changes let us introduce some notation. For $|\alpha|=1$ let $\cV_\alpha:L^2(\mu)\to L^2(\mu_\alpha)$ 
%(or should we use the notation $\Phi^*_\alpha$ to be consistent with the Clark operator?) 
be the unitary operator intertwining the operator $U_\alpha$ (acting in $L^2(\mu)$) and its spectral representation $M_\xi$ in $L^2(\mu_\alpha)$, 
\[
\cV_\alpha U_\alpha = M_\xi \cV_\alpha. 
\]
As it is customary in the perturbation theory we assume that all the spectral measures $\mu_\alpha$ are the spectral measures corresponding to the vector $b$ (recall that $b(\xi)\equiv 1$ in $L^2(\mu)$), which means that $\cV_\alpha b= \1$. Note that then $\cV_\alpha (\alpha b_1) (\xi) \equiv \overline\xi$. 

To compute $\Phi_{\alpha, \gamma}$  let us begin by making two observations. First, 
\[
\Phi_{\alpha, \gamma} = \cV_\alpha \Phi_\gamma \,,
\]
or, equivalently
\begin{align}
\label{Phi^*_{alpha,gamma}}
\Phi_{\alpha, \gamma}^* =\Phi_\gamma^* \cV_\alpha^*  \,.
\end{align}
The second observation is that an appropriately interpreted  ``universal'' representation formula (Theorem \ref{t-reprB}) gives us a 
formula for $\Phi_{\alpha, \gamma}^*$. Namely, in the spectral representation of the operator $U_\alpha$ (i.e.~on $L^2(\mu_\alpha)$) we can write 
\[
M_\xi +  (\gamma/\alpha - 1) b^\alpha (b_1^\alpha)^* = \cV_\alpha U_\gamma \cV_\alpha^* 
= \cV_\alpha \Phi_\gamma \cM_{\theta_\gamma} \Phi_\gamma^* \cV_\alpha^*
= \Phi_{\alpha,\gamma} \cM_{\theta_\gamma} \Phi_{\alpha,\gamma}^* \,, 
\]
where $b^\alpha =\cV_\alpha b$, $b_1^\alpha = \cV_\alpha (\overline\alpha b_1)$.

But then we can apply Theorem \ref{t-reprB} to get that if $c^{\alpha,\gamma} = \Phi_{\alpha,\gamma}^* b^\alpha$ and 
$c^{\alpha,\gamma}_1 = \Phi_{\alpha,\gamma}^* b^\alpha_1$ then $\Phi_{\alpha, \gamma}^*$ is given by \eqref{f-reprB} with $c^{\alpha,\gamma}$ and $c_1^{\alpha,\gamma}$ instead of $c^\gamma$ and $c_1^\gamma$ respectively in the definition of $A$ and $B$. 

But using \eqref{Phi^*_{alpha,gamma}} we get that 
\[
c^{\alpha, \gamma} = \Phi_{\alpha,\gamma}^* b^\alpha = \Phi^*_\gamma \cV_\alpha^* b^\alpha = \Phi_\gamma^* b = c^\gamma, 
\]
and 
\[
c_1^{\alpha, \gamma} =\Phi_{\alpha,\gamma}^* b_1^\alpha = \Phi^*_\gamma \cV_\alpha^* b_1^\alpha 
=\overline \alpha\Phi^*_\gamma  b_1 = \overline \alpha c_1^\gamma. 
\]
Therefore, to get the formula for $\Phi_{\alpha, \gamma}^*$ with $ \Phi_{\alpha, \gamma}^* b^\alpha = c^\gamma$ (i.e.~such that $\Phi_{\alpha,\gamma}^* \1 = c^\gamma$) one just has to replace in \eqref{f-reprB} $\mu$ by $\mu_\alpha$, 
and $c_1^\gamma$ by $\overline\alpha c_1^\gamma$ ($c^\gamma$ remains the same). 

Let us now consider the Sz.-Nagy--Foia\c{s} transcription. One of the ways to get the formula for $\Phi_{\alpha,\gamma}^*$ would be to take the ``universal formula'' above and then repeat the proof of Theorem \ref{t:AltRepr}. Alternatively, one could compute a formula for the adjoint $\cV_\alpha^\ast$ and use the composition \eqref{Phi^*_{alpha,gamma}}. In fact in Section \ref{s:V_alpha} below we provide a representation for $\cV_\alpha$. Or one can consider the following approach, allowing to almost avoid any calculations. 

Namely, let us look at the formula \eqref{U_alpha_gamma}: The operator $U_\gamma$ is represented here as a perturbation of the operator $U_\alpha$ by the rank one operator $(\gamma/\alpha - 1) b ( \overline\alpha b_1)^*$. 

Note that in the spectral representation of the operator $U_\alpha$ the operator $U_\gamma$ is given by 
\begin{align}
\label{U_alpha_gamma_01}
M_\xi + (\gamma/\alpha - 1) b^\alpha (b_1^\alpha)^*, 
\end{align}
where, recall, $b^\alpha = \cV_\alpha b$, $b_1^\alpha = \overline \alpha \cV_\alpha b_1$, $b^\alpha =\1$, $b_1^\alpha(\xi) \equiv \overline \xi$, $\xi\in \T$. 

Let us compute the characteristic function $\theta^{\alpha}_{\gamma/\alpha}$ of the above operator \eqref{U_alpha_gamma_01} with $b_1^\alpha$ and $b^\alpha$ taken for the basis vectors in the corresponding defect subspaces.  

Then we are in the situation of Section \ref{s:CharFunctU_gamma} with $\mu_\alpha$ instead of $\mu=\mu_1$. 
Therefore the characteristic function $\theta^{\alpha}_{\gamma/\alpha}$ is given by the formulas from Lemmas 
\ref{l:CharFunct-U_0},  \ref{CharFunct-U_gamma} with $\mu_\alpha$ instead of $\mu$ and $\gamma/\alpha$ instead of $\gamma$. 
This gives us a recipe for obtaining formulas for $\theta^\alpha_{\gamma/\alpha}$ from the corresponding formulas  for $\theta_\gamma$. In particular
\[
1-\theta^\alpha_{0}(z) = 1/R\mu_\alpha(z), \qquad z\in\D. 
\] 

On the other hand we know that $\theta^\alpha_{\gamma/\alpha}$ is the characteristic function of the operator $U_\gamma$ with vectors $\overline\alpha b_1$ and $b$ taken for the basis vectors in the corresponding defect subspaces. Since in the definition of $\theta_\gamma$ we took $b_1$ and $b$ for the basis vectors, we can conclude that 
\[
\theta^\alpha_{\gamma/\alpha} = \overline \alpha\theta_\gamma, \qquad \text{and} \qquad \Delta_{\alpha, \gamma} = \Delta_\gamma. 
\]
Combining this formula with the previous identity we get, in particular, that 
\begin{align}
\label{R_mu_alpha}
1- \overline \alpha\theta_0 = 1-\theta^\alpha_{0}(z) = 1/R\mu_\alpha(z), \qquad z\in\D. 
\end{align}

Just to check, the  identity
\[
\theta^\alpha_{\gamma/\alpha} = \frac{\theta^\alpha_0 -\gamma/\alpha}{1-\overline\gamma/\overline\alpha \theta^\alpha_0}
\]
is obtained from \eqref{theta_0-theta_gamma} by the appropriate replacements; and it is an easy calculation to verify that the above identity is equivalent to \eqref{theta_0-theta_gamma}.   

Now let us give the representation of the operator $\Phi_{\alpha, \gamma}^*$ adapted to the Sz.-Nagy--Foia\c{s} transcription, i.e.~the analog of Theorem \ref{t:AltRepr}. Let $T^\alpha_+f$ denote the non-tangential boundary values of $Rf\mu_\alpha(z)$, $z\in\D$. Then replacing in Theorem \ref{t:AltRepr} $\gamma$ by $\gamma/\alpha$, $\theta_\gamma$ by $\theta^\alpha_{\gamma/\alpha} = \overline \alpha \theta_\gamma$, and $\theta_0 $ by $\theta_{\alpha, 0} =\overline\alpha \theta_0$, and $T_+$ by $T_+^\alpha$ we get the following representation for 
$\Phi_{\alpha,\gamma}^*$:  
\begin{align}
\label{AltRepr-alpha}
 (1-|\gamma|^2)^{1/2}
 wt \Phi_{\alpha,\gamma}^* f
 & = 
\left( \begin{array}{c} 0 \\  (\overline\gamma/\overline\alpha - (\overline\gamma/\overline\alpha - 1) T^\alpha_+ \1 )\Delta_\gamma \end{array}
\right) f 
+
\left( \begin{array}{c} (1 + \overline\gamma \theta_\gamma) / T^\alpha_+1 \\  (\overline\gamma/\overline\alpha - 1) \Delta_\gamma \end{array}
\right) T^\alpha_+ f 
\\
\notag
 & = 
\left( \begin{array}{c} 0 \\   \frac{1-\overline\gamma\theta_0}{|1-\overline\gamma\theta_0|}   T_+^\alpha \1 \cdot\Delta_0 \end{array}
\right) f 
+
\left( \begin{array}{c} \frac{1-|\gamma|^2}{1 - \overline\gamma \theta_0} \cdot\frac1{ T^\alpha_+\1} \\  (\overline\gamma /\overline\alpha - 1) \frac{(1-|\gamma|^2)^{1/2}}{|1-\overline\gamma\theta_0|} \Delta_0 \end{array}
\right) T_+^\alpha f \,. 
\end{align}
Note that in the first (top) entry of the formula $\alpha$ in the coefficients cancels out, and in the second entry it does not. 

But the above formula \eqref{AltRepr-alpha} is not yet the formula we are looking for! To get it we applied Theorem \ref{t:AltRepr} with $\mu_\alpha$ instead of $\mu$ and $\theta^\alpha_{\gamma/\alpha} = \overline\alpha\theta_\gamma$ instead of $\theta_\gamma$. But that means that the result in the right hand side of \eqref{AltRepr-alpha} belongs to $\cK_{\overline\alpha \theta}$. So the formula \eqref{AltRepr-alpha} is an absolutely correct formula giving the representation of the operator $\Phi^*_{\alpha,\gamma}$ in the model space $\cK_{\overline\alpha\theta_\gamma}$; that is why we used $\wt \Phi^*_{\alpha,\gamma}$ and not $\Phi^*_{\alpha,\gamma}$ there.

To get the representation with the model space $\cK_{\theta_\gamma}$ we notice that the map
\[
\left( \begin{array}{c} g_1\\ g_2 \end{array} \right) 
\mapsto 
\left( \begin{array}{c} g_1\\ \overline\alpha g_2 \end{array} \right)
\]
is a unitary map from $\cK_{\overline\alpha\theta_\gamma}$ onto $\cK_{\theta_\gamma}$. Moreover, it maps the defect vector $c$ given by equation \eqref{c} for the space $\cK_{\overline\alpha\theta_\gamma}$ to the corresponding defect vector $c$ for the space $\cK_{\theta_\gamma}$. Therefore, to obtain the representation formula for $\Phi_{\alpha,\gamma}^*$ we need to multiply the bottom entries in \eqref{AltRepr-alpha} by $\overline\alpha$, which gives us 
\begin{align}
\label{AltRepr-alpha-01}
 (1-|\gamma|^2)^{1/2}
 \Phi_{\alpha,\gamma}^* f
 & = 
\left( \begin{array}{c} 0 \\  (\overline\gamma - (\overline\gamma - \overline\alpha) T^\alpha_+ \1 )\Delta_\gamma \end{array}
\right) f 
+
\left( \begin{array}{c} (1 + \overline\gamma \theta_\gamma) / T^\alpha_+1 \\  (\overline\gamma -\overline\alpha) \Delta_\gamma \end{array}
\right) T^\alpha_+ f 
\\
\notag
 & = 
\left( \begin{array}{c} 0 \\  \overline\alpha \frac{1-\overline\gamma\theta_0}{|1-\overline\gamma\theta_0|}   T_+^\alpha \1 \cdot\Delta_0 \end{array}
\right) f 
+
\left( \begin{array}{c} \frac{1-|\gamma|^2}{1 - \overline\gamma \theta_0} \cdot\frac1{ T^\alpha_+\1} \\  (\overline\gamma  - \overline\alpha ) \frac{(1-|\gamma|^2)^{1/2}}{|1-\overline\gamma\theta_0|} \Delta_0 \end{array}
\right) T_+^\alpha f \,. 
\end{align}
Just to check: we should have $\Phi^*_{\alpha,\gamma} \1=c$, where $c$ is given by \eqref{c} with $\theta=\theta_\gamma$, and we get this in \eqref{AltRepr-alpha-01}. Indeed, we get from the first line in \eqref{AltRepr-alpha-01} that 
\[
\Phi_{\alpha,\gamma}^* \1 =
\left( \begin{array}{c} 1 + \overline\gamma \theta_\gamma \\  \overline\gamma  \Delta_\gamma \end{array}
\right)
\]
which coincides with \eqref{c} if we recall that $\theta_\gamma(0) =-\gamma$.

\subsection{Connection with the Clark's construction}
\label{us-vs-clark}

D.~Clark in  \cite{Clark} approached to the problem from a different point of view. He started from a model operator $\cM_\theta$ (for an inner $\theta$), considered all its unitary perturbations of rank one and computed spectral measures of these extensions. Let us compare his formulas  with ours. 

Let us start with $\theta \in H^\infty$, $\|\theta \|_\infty\le 1$, and consider all rank one unitary perturbations of the model operator $\cM_\theta$. All such operators are parametrized by a complex parameter $\alpha$, $|\alpha|=1$. Namely, any such perturbation $V_\alpha$ acts as the model operator (equivalently multiplication by $z$) on $\cK_\theta \cap (c_1)^\perp$ and  $V_\alpha c_1 = \alpha c$; here $c$ and $c_1$ are the defect vectors defined by \eqref{c}, \eqref{c_1}. This was exactly the setup considered by Clark in \cite{Clark} for the inner function $\theta$; our $\alpha$ corresponds to his parameter $w$.

In our model we had $\theta(0)=-\gamma$, so let us define $\gamma:=-\theta(0)$. Then clearly 
\[
\cM_\theta = V_1 + (\gamma-1) c (c_1)^*. 
\]
Therefore, if we construct a measure $\mu=\mu_1$ such that $\theta$ is represented by $\theta_\gamma$ in Lemma \ref{CharFunct-U_gamma}, this measure will be the spectral measure of the operator $V_1$, and $V_1\Phi_\gamma^* = \Phi_\gamma^* U_1$. 
%$\omega_1(\gamma), \omega_\gamma(1)$

We can rewrite the formula in Lemma \ref{CharFunct-U_gamma} as 
\begin{align}
\label{theta_gamma-02}
\theta_\gamma =\omega \frac{R_2\mu - \beta}{R_2\mu +\overline\beta}
\end{align}
where $\omega = \omega_1(\gamma) =\frac{1-\gamma}{1-\overline\gamma}$ (more generally $\omega_{\alpha}(\gamma) = \frac{\alpha-\gamma}{1-\overline\gamma \alpha }$, $\alpha\in\T$),  and $\beta=\frac{1+\gamma}{1-\gamma}$\,.% 
\footnote{Note that in \cite{Clark} a different notation is used for the parameters:  
%parameters are exchanged:
our $\alpha$ corresponds to the parameter $w$ in \cite{Clark} and  our $\omega_\alpha(\gamma)$ corresponds to  $\alpha=\alpha_w$ in \cite{Clark}. }

From \eqref{theta_gamma-02} we get that 
\[
R_2\mu =\frac{\omega\beta + \overline\beta \theta_\gamma}{\omega-\theta_\gamma} =
 \frac{\beta + \overline\beta \overline\omega \theta_\gamma}{1-\overline\omega\theta_\gamma},  
\]
and 
\[
\re (R_2\mu) = \re(\beta) \frac{1-|\theta|^2}{|1-\overline\omega\theta|^2} = \frac{1-|\gamma|^2}{|1-\gamma|^2}  \frac{1-|\theta|^2}{|1-\overline\omega\theta|^2}. 
\]

Since for $\xi\in\T$, $z\in\D$
\[
\re \frac{1+\overline\xi z}{1-\overline\xi z} = \frac{1-|z|^2}{|1-\overline\xi z|^2}
\]
is the Poisson kernel, the measure $\mu=\mu_1$ is the measure whose Poisson extension to the unit disc gives 
\[
\re \frac{\beta + \overline\beta \overline\omega \theta_\gamma}{1-\overline\omega\theta_\gamma} 
= \frac{1-|\gamma|^2}{|1-\gamma|^2}  \frac{1-|\theta|^2}{|1-\overline\omega\theta|^2}.
\]

In Clark's construction his measure (let us call it $\wt\mu=\wt\mu_1$) was defined as the measure whose Poisson extension to the unit disc is 
\[
\re  \frac{1+\overline\omega  \theta}{1-\overline\omega\theta} =  \frac{1-|\theta|^2}{|1-\overline\omega\theta|^2}, 
\]
so
$\mu= (1-|\gamma|^2)|1-\gamma|^{-2}\wt\mu$ (note that $\mu=\wt\mu$ for $\gamma=0$).  

For an inner $\theta$ the adjoint Clark  operator was defined in \cite{Clark} as 
\begin{align}
\label{Phi^*-Clark}
f\mapsto (1-\overline\omega \theta) R f\mu. 
\end{align}
From \eqref{NCT-01} we get using $1-\theta_0=R\mu$ that 
for inner $\theta=\theta_\gamma$ 
\[
\Phi_\gamma^* f = \frac{1-\gamma}{(1-|\gamma|^2)^{1/2}}(1-\overline\omega \theta) Rf\mu, 
\]

If $\gamma=-\theta(0)=0$ our formula coincides with one presented in \cite{Clark}. If $\gamma\ne 0$ the two formulas differ by a constant factor $(1-\gamma)(1-|\gamma|^2)^{-1/2}$; its modulus compensates for the different normalization of the measures. 

To get the spectral measures $\mu_\alpha$ in our model we need to replace $\theta_\gamma$ by $\theta^\alpha_\gamma= \overline\alpha \theta_\gamma$. Since
\[
\omega_\alpha(\gamma) := \frac{\alpha-\gamma}{1-\overline\gamma\alpha} = \alpha \frac{1-\gamma\overline\alpha}{1-\overline\gamma \alpha} = \alpha \omega_1(\gamma\overline\alpha), 
\]
this replacement is equivalent to using $\omega_\alpha(\gamma)$ for $\omega$, which is exactly the way to get $\wt\mu_\alpha$ in the Clark construction. Therefore again, we have $\mu_\alpha= (1-|\gamma|^2)|1-\gamma|^{-2}\wt\mu_\alpha$. 

The adjoint Clark operator $\wt\Phi_{\alpha,\gamma}^*:L^2(\wt\mu_\alpha)\to \cK_\theta$ from the Clark's  construction is given by 
\begin{align}
\label{Phi^*-Clark-02}
\wt\Phi^*_{\alpha,\gamma} f = (1-\overline{\omega} \,\theta) Rf\wt\mu_\alpha, 
\qquad \omega= \omega_\alpha(\gamma). 
\end{align}

Looking at the top entry in \eqref{AltRepr-alpha-01} we get that in our construction
\[
 \Phi_{\alpha,\gamma}^* f = \frac{1-\gamma\overline\alpha}{(1-|\gamma|^2)^{1/2}}(1-\overline\omega \theta) Rf\mu,
\qquad \omega= \omega_\alpha(\gamma). 
\]
Again the two formulas differ by a constant factor $(1-\gamma\overline\alpha) (1-|\gamma|^2)^{-1/2}$, whose absolute value compensate for the different normalization of the measures. 

\begin{rem}
In \cite{SAR} D.~Sarason constructed a unitary operator between $H^2(\mu)= \spa\{z^n: n\in\Z_+\}$ and the de Branges space $\cH(\theta)$. Like D.~Clark, he started with the model space $\cH(\theta)$ (he used $b$ for $\theta$) and then obtained the measure $\mu$ and the corresponding unitary operator. 

In his construction the measure $\mu$ coincides with the measure $\wt\mu_\alpha$ in the Clark's construction, where $\alpha= (1+\gamma)(1+\overline\gamma)^{-1}$, so $\omega=\omega_\alpha(\gamma) =1$. 
His operator is  given by \eqref{Phi^*-Clark-02} with $\alpha$ and $\omega$ as above.  

If $\theta$ is an extreme point of the unit ball in $H^\infty$, then $H^2(\mu)=L^2(\mu)$ and there is a canonical isomorphism (see Remarks \ref{r:dBR_model} and \ref{r:dBr-R_orig} above)  between the space  $\cH(\theta)$ and the model space $\cK_\theta$ in the de Branges--Rovnyak transcription. Thus, in \cite{SAR} a description of the adjoint Clark operator was obtained in the case of $\theta $ being an extreme point.

\end{rem}

\section{Appendix: unitary operators intertwining spectral representations  of unitary rank \texorpdfstring{$1$}{1}  perturbations}

\label{s:V_alpha}

A natural question would be to describe an operator intertwining $U_\alpha$, $|\alpha|=1$ and its spectral representation, the multiplication $M_z$ by $z$ in $L^2(\mu_\alpha)$, i.e.~to describe a unitary operator $\cV_\alpha :L^2(\mu)\to L^2(\mu_\alpha)$ such that 
\begin{align}
\label{ComRel-02}
\cV_\alpha U_\alpha   = M_z \cV_\alpha .
\end{align}
Here, recall $U=U_1$ is the multiplication $M_\xi$ by the independent variable $\xi$ in $L^2(\mu)$, and $U_\alpha= U+(\alpha-1)bb_1^*$, $b\equiv\1$, $b_1 = U^* b \equiv \1$, and $\mu_\alpha$ is the spectral measure of $U_\alpha$ corresponding to the cyclic vector $b$. 

This problem is  easier than describing the adjoint Clark operator, and the authors new the answer for some time. 
In \cite{mypaper} the authors treated the case of self-adjoint rank one perturbations, and derived a formula for the unitary operator corresponding to $\cV_\alpha$. Transferring the proofs from \cite{mypaper} to the unitary settings one immediately gets the desired description. In fact, some of the proofs will be easier in the unitary case, because one does not have to worry about unbounded operators; also, since in the unitary case we work on the unit circle, one does not have to worry about singularity of the kernel at $\infty$. 

Here for the readers' convenience we present the complete the  formulas for $\cV_\alpha$ in the case of rank one unitary perturbations. There are now new ideas here (comparing with \cite{mypaper} and with what was done here for adjoint Clark operators), so we do not present the proofs here: we only state the results and give a very brief outline of the proofs.

%Finally, as a certain converse to the representation theorem, we state a rigidity theorem; 
%beginning with a representation formula we derive a rank one unitary perturbation setting.

\begin{theo}[Representation Theorem]\label{t-repr-V}
Let $\cV_\alpha: L^2(\mu)\to L^2(\mu_\alpha)$ be a unitary operator satisfying \eqref{ComRel-02} and such that $\cV_\alpha \1 =\1$ (which means that $\mu_\alpha$ is the spectral measure of $U_\alpha$ corresponding to the cyclic vector $b$, $b(\xi)\equiv\1$). Then 
%of $U_\alpha$ is given by% the formula
%%
\begin{equation}
\label{repr-V}
\cV_\alpha f (z)= f(z) +(1-\alpha)\int_\T \frac{f(\xi)-f(z)}{1-\bar\xi z}\,d\mu(\xi)\qquad
\text{for all }f\in C^1(\T).
\end{equation}
\end{theo} 

As the following proof merges tools from \cite{mypaper} with methods from Section \ref{s:UnivRepr}, we omit details and sketch only the main steps.

\begin{proof}[Main steps of the proof of Theorem \ref{t-repr-V}]
Recalling that 
\[
U_\alpha =U_1 + (\alpha-1) b  b_1^\ast= M_\xi + (\alpha-1) b b_1^\ast, \qquad b(\xi)\equiv \1, \ b_1(\xi)\equiv \overline\xi,
\]
$M_\xi f(\xi) =\xi f(\xi)$, we get from 
the intertwining relationship \eqref{ComRel-02} that
%$\cV_\alpha U_\alpha = M_\xi \cV_\alpha$, 
%together with and $\cV_\alpha \ID = \ID$ implies
\[
\cV_\alpha U_1 = M_z \cV_\alpha +(1-\alpha) (\cV_\alpha b) b_1^\ast.
\]
Inductively one can show that
\[
\cV_\alpha U_1^n = M_z^n \cV_\alpha +(1-\alpha) \sum_{k=1}^n M_z^{k-1} (\cV_\alpha b)  \left((U_1^*)^{n-k} b_1\right)^\ast.
\]
Applying this formula to the function $b\equiv \1\in L^2(\mu)$ and recalling that $(U_1^n b)(\xi)=\xi^n$, $\cV_\alpha b=\1$, $(U_1^*)^{n-k} b_1\equiv \xi^{n-k+1}$ we obtain summing the geometric series
\[
(\cV_\alpha \xi^n )(z)= z^n +(1-\alpha)\int_\T \frac{\xi^n-z^n}{1-\bar\xi z}\,d\mu(\xi).
\]

The action of $\cV_\alpha$ on $\bar\xi^n$ is proved similarly.
Namely, the adjoint of the intertwining formula \eqref{ComRel-02} yields $\cV_\alpha U_\alpha^\ast = M_{\bar z} \cV_\alpha$, so
\[
\cV_\alpha U_1^* = M_{\overline z} \cV_\alpha + (1-\overline\alpha)(\cV_\alpha b_1) b^* .
\]
By induction we get 
\begin{align}
\label{induct-03}
\cV_\alpha (U_1^*)^n = M_{\overline z}^n \cV_\alpha + (1-\overline\alpha) \sum_{k=1}^n M_{\overline z}^{k-1}\cV_\alpha b_1 (U_1^{n-k} b)^*  . 
\end{align}
Applying the identity $\cV_\alpha U_\alpha = M_z \cV_\alpha $ to the vector $b_1$ and using the fact that $\cV_\alpha b\equiv \1$, we get that $\cV_\alpha b_1\equiv \alpha \overline z$. Then, applying the identity \eqref{induct-03} to the vector $b\equiv \1\in L^2(\mu)$ and summing the geometric progression we get that \eqref{repr-V} holds for $f(\xi)\equiv \overline\xi^n$.

%and with the same method one can show that
%\[
%(\cV_\alpha \bar\xi^n )(z)= \bar z^n +(1-\alpha)\int_\T \frac{\bar\xi^n-\bar z^n}{1-\bar\xi z}\,d\mu(\xi).
%\]

Thus, \eqref{repr-V} holds for all trigonometric polynomials $f$. The same approximation reasonings as in the proof of Theorem \ref{t-reprB} give the result for all $f\in C^1$. 
%By a simple approximation argument (in analogy to the reasoning at the end of the proof of 
%Theorem \ref{t-reprB}) the proposition follows.
\end{proof}

Let $T_r=T_r^{1,\alpha}:L^2(\mu)\to L^2(\mu_\alpha)$ be the integral operators with kernel $1/(1- r\overline\xi z)$, $r\in\R_+\setminus\{1\}$; we use indices $1,\alpha$ to indicate the measures in the domain ($\mu=\mu_1$) and in the target space. Acting exactly as in Section \ref{s:SIO-reg} we can show that the operators $T_r^{1,\alpha}$ with the kernel $1/(1- r\overline\xi z)$ are uniformly (in $r\in\R_+\setminus\{1\}$) bounded. In Section \ref{s:SIO-reg} we proved this fact for the measure $v|dz|$ instead of $\mu_\alpha$, but for the proof to work we only need that the measures $\mu$ and $\mu_\alpha$ do not have common atoms. And that is definitely the case, since by the Aronszajn--Donoghue theorem the singular parts of $\mu$ and $\mu_\alpha$ are mutually singular. In fact, we do not need the full Aronszajn--Donoghue theorem here: we only need the fact that the operators $U_1$ and $U_\alpha$ do not have common eigenvalues, which is a simple exercise (recall that $b$ is a cyclic vector for $U_1$).

Since the operators $T^{1,\alpha}_r$ are uniformly bounded, we can consider the limits (in the weak operator topology)
\begin{align}
\label{WOT-lim-02}
T^{1,\alpha}_\pm := \text{w.o.t-}\lim_{r\to 1^{\mp}} T^{1,\alpha}_r
\end{align}

Note, that uniform boundedness of $T^{1,\alpha}_r$ is not enough for the existence of the limit, it only implies the existence of  accumulation points (in the weak operator topology). However, the existence of boundary values imply, see Proposition \ref{p:wot}, the existence of the limit. Namely, one can see from the proof of Proposition \ref{p:wot}  that if $\lim_{r\to 1^\mp} T_r^{1,\alpha} f =: T^{1,\alpha}_\pm f$ $\mu_\alpha$-a.e., then $\text{w-}\lim_{r\to 1^\mp} T_r^{1,\alpha} f =: T^{1,\alpha}_\pm f$.  

The existence of boundary values follows from Proposition \ref{p:BV-02} below. 

\begin{prop}
\label{p:BV-02}
Let $f\in L^2(\mu)$, and let 
\[
F(z) = R f\mu(z) = \int_\T \frac{f(\xi) }{1-\overline \xi z} d\mu(\xi), \qquad |z|\ne 1. 
\]
Then non-tangential boundary values of $F(z)$ as $|z|\to 1^\mp$ exist $\mu_\alpha$-a.e.

In particular, the limits
\[
\lim_{r\to 1^\mp} T_r^{1,\alpha} f(z)
\]
exist $\mu_\alpha$-a.e.
\end{prop}

\begin{proof}
Existence of the boundary values with respect to the Lebesgue measure (and so with respect to the absolutely continuous part of $\mu_\alpha$) follows from the classical results. 

To prove the existence of the boundary values $(\mu_\alpha)\ti s$-a.e.~we will use Poltoratskii's theorem. 

Take $f\in L^2(\mu)$ and let $f_\alpha := \cV_{\alpha} f $, $f_\alpha\in L^2(\mu_\alpha)$ (note, that as an abstract operator $\cV_\alpha$ is well defined. Define $g\in \cK_{\theta_0}$,  $g:= \Phi^*_0 f$. 

Then clearly, see \eqref{Phi^*_{alpha,gamma}}
\[
g:= \Phi^*_{\alpha,0} f_\alpha. 
\]
Considering the Sz.-Nagy--Foia\c{s} transcription and comparing formulas for $g_1$ we get, see \eqref{AltRepr-alpha-01} and \eqref{R_mu_alpha},
\[
(1-\theta_0(z)) R f\mu(z) = (1-\overline \alpha \theta_0(z)) R f_\alpha \mu_\alpha (z), \qquad z\in \D.
\]
Therefore 
\[
R f\mu(z) = (1-\theta_0(z))^{-1} \frac{R f_\alpha\mu_\alpha (z)}{R \mu_\alpha (z)}\,, \qquad z\in\D.
\]
(Note that the latter equality can also be deduced from the standard resolvent identities for rank one perturbations, see \cite{AbaPolt}.)

By Polotratkii's theorem, see \cite[Theorem 2.7]{NONTAN}, the non-tangential boundary values of $Rf_\alpha\mu_\alpha/R\mu_\alpha$ exist (and coincide with $f_\alpha$) $(\mu_\alpha)\ti s$-a.e. On the other hand, $(\mu_\alpha)\ti s$-a.e.~the non-tangential boundary values of $\theta_0$ exist and equal $\alpha$ (this immediately follows from the fact that $R\mu_\alpha = (1 + \overline \alpha \theta_0)/(1-\overline\alpha\theta_0)$). Thus, for $|z|<1$ the non-tangential boundary values of $Rf\mu$ exist (and equal $\overline\alpha f_\alpha$) $(\mu_\alpha)\ti s$-a.e. 

To treat the case $|z|>1$ let us denote $w:=1/\overline z$. Since
\[
\frac1{1-\overline \xi z}= \frac{-\xi \overline w}{1- \xi \overline w} = 1 - \frac{1}{1- \xi \overline w}, 
\]
we conclude that 
\[
R f \mu (z) = \int_\T f d\mu - \overline{\left(R \overline f  \mu \right) (w)}, 
\]
and the existence of the non-tangential boundary values follows from the case $|z|<1$. 
\end{proof}

%In the spirit of Section \ref{s:AltRepr} we use singular integral operators to find 
%two alternative representation formulas, both of which hold on all of $L^2(\mu)$.

%Recall that $T_+ f$ and $T_- f$ ($f\in L^1(\mu)$) are the non-tangential boundary value of the Cauchy integral $(Rf\mu)(z) = \int_{\T} \frac{f(\xi)d\mu(\xi)}{1-\overline{\xi} z }$ as $z$ approaches $\T$ from inside and outside of the disc respectively.

%By Aronszajn--Donoghue theory the singular parts of $\mu$ and $\mu_\alpha$ are mutually singular. 
%This allows us to proceed in analogy to Section \ref{s:AltRepr} while replacing $v dm$ by $d\mu_\alpha$, 
%and the statement in Proposition \ref{p:KernelPhi*} by the appropriate single-valued statement.

%The spectral representation $\cV_\alpha:L^2(\mu)\to L^2(\mu_\alpha)$ is given by the 
%alternative representation formulas

Replacing the kernel in  \eqref{repr-V} by $1/(1-r\overline\xi z)$ and  taking the limit as $r\to 1^\mp$,  
we get  different formulas for  $\cV_\alpha$,
\begin{align}
\label{repr-V-01}
 \cV_\alpha f  &= [\1 - (1-\alpha)T_\pm^{1,\alpha}\ID] f +(1-\alpha)T_\pm^{1,\alpha} f; 
\end{align}
here $T^{1,\alpha}_\pm$ are defined above in  \eqref{WOT-lim-02}.

% and
%\begin{align*}
% \cV_\alpha f (z)  &= f(z)[\ID - (1-\alpha)T_-\ID] +(1-\alpha)T_- f
%\end{align*}
%for $f\in L^2(\mu)$.

%\begin{rem*}
%Although functions from $L^2(\mu)$ are not necessarily defined a.e.~wrt $(\mu_\alpha)\ti{s}$, 
%this formula still makes sense. Indeed, as $T_+\ID = (1-\alpha)^{-1}$ a.e.~wrt $(\mu_\alpha)\ti{s}$, 
%the term in square brackets vanishes there. The theory of singular integral operators 
%from Section \ref{s:AltRepr} grants precise meaning to $T_+f$.
%\end{rem*}

Legality of taking the  limit is justified exactly  the same as in the beginning of the proof of Theorem \ref{t:AltRepr}.

\begin{rem}
In fact, it is not necessary to use the full power of Poltoratskii's theorem to prove the weak convergence. In fact, one can, as it was done in the proof of Theorem 3.2 in \cite{mypaper}, use  the existence $(\mu_\alpha)\ti s$-a.e.~of non-tangential boundary values of $R\mu = 1/(1-\theta_0)$ to get the $(\mu_\alpha)\ti s$-a.e.~existence of the boundary values of $R f\mu$ for $f\in C^1$. This would imply the weak convergence of $T_r^{1,\alpha} f$ on a dense set of $C^1$ functions, and thus the convergence of $T_r^{1,\alpha}$ in the weak operator topology. 
\end{rem}

%%%%%%%%%%%%%%%%%%%%%%%%%%%%%%%%%%%%%
\subsection{Rigidity Theorem}
The following rigidity result can be understood as a converse to the  Representation Theorem (Theorem \ref{t-repr-V}).% The proof is exactly as the proof of Theorem 2.2 in \cite{mypaper} (with obvious modifications/simplifications), so we do not present it here. 
Corresponding result for self-adjoint perturbations was proved by the authors in \cite{mypaper}, see Theorem 2.2 there. 

\begin{theo}[Rigidity Theorem]
\label{rigTM}
Let a probability measure $\mu$ on $\T$ be supported on at least two distinct points.
Let $\alpha\in \T\setminus \{1\}$, and let $\cV f$ be defined for $C^1$ functions $f$ by the right hand side of  \eqref{repr-V}.

Assume $\cV$ extends to a bounded operator from $L^2(\mu)$ to $L^2(\nu)$ and assume $\Ker \cV=\{0\}$.

Then there exists a function $h$ such that $1/h\in L^\infty(\nu)$, and $M_h \cV$ is a unitary operator from $L^2(\mu)\to L^2(\nu)$ (equivalently, that $\cV:L^2(d\mu)\to L^2( |h|^{2}\,d\nu)$ is unitary). 

Moreover, the measure $|h|^2 \nu$ is exactly the Clark measure $\mu_\alpha$ from Proposition \ref{p:BV-02}, and $\cV$ treated as the operator $L^2(\mu)\to L^2(\mu_\alpha)$ is exactly the operator $\cV_\alpha$ from that proposition. 
%
%Moreover, the unitary operator $\cU:=M_h\cV$ gives the spectral representation of the unitary 
%operator $U_\alpha := M_z + (1-\alpha) (\fdot, \bar\xi)\ID$ in $L^2(\mu)$, namely 
%$\cU U_\alpha = M_\zeta \cU$, where $M_\xi$ is the multiplication by the independent variable $\xi$ in $L^2(\nu)$. 
\end{theo}

\begin{proof}%[Proof]
The proof of Theorem \ref{rigTM} follows the proof of Theorem 2.2 in \cite{mypaper}. 

Let $\xi$ and $z$ denote the independent variables in $L^2(\mu)$ and  $L^2(\nu)$ respectively, and let $M_\xi:L^2(\mu)\to L^2(\nu)$ and  $M_z:L^2(\nu)\to L^2(\nu)$ be the multiplication operators by the independent variable in the corresponding spaces. We will also use symbols $\1_\xi$ and $\1_z$ to indicate that we consider the function $\1$ as an element of $L^2(\mu)$ or $L^2(\nu)$. 
%:f(z)\mapsto z f(z)$. Similarly, we use $\xi$ and $M_\xi$ for $L^2(\mu)$.

The first step is to show that operator $\cV\cV^*$ commutes with $M_z$, i.e.
\begin{align}
\label{commutes}
\cV\cV^* M_z = M_z \cV\cV^*.
\end{align}

For $f=b\equiv\ID_\xi$, formula \eqref{repr-V} yields $\cV b= \cV\ID_\xi=\ID_z$.  Denote $b_1 = M^*_\xi$, $b_1(\xi)\equiv \overline \xi$. Applying \eqref{repr-V} to $f(\xi)= \xi^n$ and $f(\xi)=\xi^{n+1}$ and subtracting we get  for  $f(\xi)=\xi^n$, $n\in\Z$:
\begin{align*}
(M_z\cV-\cV M_\xi)f(z)
&=(1-\alpha)\int \frac{z\xi^n-\xi^{n+1}}{1-\bar\xi z} \dd\mu(\xi)
 =(\alpha-1)\int \xi^{n+1}\dd\mu(\xi)
 \\&=(\alpha-1) (\xi^n,\bar\xi)\ci{L^2(\mu)}\ID_z
= (\alpha-1) (f,b_1)\ci{L^2(\mu)}\cV b;
\end{align*}
here in the last equality we used the fact that $\cV b=\1_z$. 

By linearity we can extend this formula to polynomials and then using standard approximation reasoning extend it to all of $L^2(\mu)$. %Together with $\cV\ID_\xi=\ID_z$ we have
Moving $\cV M_\xi$ to the right hand side we rewrite the identity as
\begin{align}
\label{intertwining}
M_z\cV = \cV U_\alpha\qquad \text{where}\qquad U_\alpha:= M_\xi +(\alpha-1)b b_1^*
%(\fdot, \bar\xi)\ID_\xi 
\quad \text{ on }L^2(\mu).
\end{align}
Taking the adjoint gives $\cV^*M_z=U_\alpha \cV^*$ and so we have:
$$
M_z\cV\cV^*=\cV U_\alpha \cV^* = \cV\cV^*M_z
$$
and obtained the desired commutation relation \eqref{commutes}.

Next we want to prove that 
\begin{align}
\label{kernel}
\Ker \cV^*=\{0\}.
\end{align}
Since $\Ker \cV^* =\Ker \cV\cV^*$, the commutation relation \eqref{commutes} implies that
%and by \eqref{kernel}, 
the kernel $\Ker \cV^*$ is a spectral subspace of $M_z$. Namely, there exists a Borel subset $E\subset \T$ such that
\[
\Ker \cV^* = \{f \in L^2(\nu)\,:\, \1\ci{\T\setminus E} f =0\}.
\]

Assuming that $\Ker \cV^*\neq \{0\}$, i.e.~that $\nu(E)>0$, let us obtain a contradiction by constructing a function $f=\1\ci{E_1}$, $E_1\subset E$, $\nu(E_1)>0$ such that $f\notin\Ker \cV^*$.

By assumption $\supp\mu$ consists of at least two points. Therefore, there exists $\tau\in[0,\pi)$ such that for the open arcs $A_1:= \{e^{it}: t\in(\tau,\tau+\pi)\}$, $A_2:= \{e^{it}: t\in(\tau-\pi,\tau)\}$, we have $\mu(A_{1,2})>0$.

By the regularity of the measure $\mu$ there exist closed arcs $I_{1,2}\subset A_{1,2}$ such that $\mu(I_{1,2})>0$.

Since $\clos A_1\cup\clos A_2 =\T$, at least one of the conditions $\nu(E\cap\clos A_1)>0$ or $\nu(E\cap\clos A_2)>0$ holds. 

Assume for the definiteness that for $E_1:=E\cap\clos A_1$ we have $\nu(E_1)>0$. Then for $z\in E_1$, $\xi\in I_2$ we have 
\[
\re  (1-\overline\xi z )\ge \delta > 0
\]
(the worst case is when $\xi$ and $z$ are closest). Since $|1-\overline\xi z|\le 2$ we conclude that 
\[
\re \left( \frac{1}{1-\overline \xi z} \right)\ge \frac{\delta}{2^2}. 
\]
Then it follows from \eqref{repr-V} that for $f=\1\ci{E_1}$, $g=\1\ci{I_2}$ we have 
\begin{align*}
 |(g, \cV^* f )| &= |1-\alpha| \cdot \left| \int_{I_2}\int_{E_1}  \frac{1}{1-\overline\xi z} d\mu(\xi)d\nu(z) \right|
\\ & \ge |1-\alpha| \re \int_{I_2}\int_{E_1}  \frac{1}{1-\overline\xi z} d\mu(\xi)d\nu(z)
\ge |1-\alpha| \frac{\delta}{4} \mu(I_2)\nu(E_1) > 0, 
\end{align*}
so $f\notin \ker \cV^*$. We got a contradiction. The case $\nu(E\cap\clos A_2)>0$ is treated absolutely the same way. 

%there exist closed arcs $I_1\subset T:=\{e^{it}: t\in(\tau,\tau+\pi)\}, I_2\subset \T\backslash\overline{T}$ with $\mu(I_1)>0, \mu(I_2)>0$. We need to consider two cases.
%
%If $\nu(E\cap T)>0$, we can pick $\eta\in(\tau,\tau+\pi)$ such that for $X:=E\cap\{e^{it}: t\in[\tau,\eta]\}$ we have $\nu(X)>0$. Let $f=\chi\ci{X}$. Then $f\in L^2(\nu)$ and $\chi\ci{\T\setminus E} f =0$. Take $g\in C^1$ such that $g|\ci{I_2}=1$ and so that $g=0$ on $\{e^{it}: t\in[\tau-\e,\tau+\pi+\e]\}$. This is possible, because $I_2\Subset\T\backslash\overline{T}$. Then $g$ and $f$ have separated compact support. We have
%\[
%(f,\cV g)\ci{L^2(\nu)}=\int\ci{X}\int\ci{I_2}\frac{f(z)\overline{g(\xi)}}{1-\bar z \xi}\,\dd\mu(\xi)\dd \nu(z)\neq0,
%\]
%since $\int\ci{I_2}\frac{\overline{g(\xi)}}{1-\bar z \xi}\,\dd\mu(\xi)\neq0$ and has values contained in a sector of angle $<\pi$ for all $z\in X$.
%
%Because $(\cV^*f,g)\ci{L^2(\mu)}=(f,\cV g)\ci{L^2(\nu)}$, we have $f\notin\Ker \cV^*$.
%
%Consider the case $\nu(E\cap T)=0$. Recall $\nu(E)>0$. So $\nu(E\cap (\T\backslash\overline{T}))>0$ and an analogous argument % as for $\nu(E\cap[a, \infty))>0$ 
%yields the desired contradiction.

Thus, we have proved that $\Ker \cV^*= \{0\}$.
% was wrong and we showed \eqref{kernel}.

For the remainder of the proof we use standard operator theoretic tools.

Let us focus on the first part of the rigidity theorem. Writing polar decomposition, we have $\cV ^* = \widetilde U |\cV ^*|$, where, recall $|\cV^*|=(\cV\cV^*)^{1/2}$ and $\widetilde U$ is a partial isometry (meaning that $\wt U$ restricted to $\ran |\cV^*|=(\ker\cV^*)^\perp $ is an  %partial 
isometry). 
%Since $\cV \cV ^*\ge0$ the operator $|\cV ^*|=(\cV \cV ^*)^{1/2}$ exists. 
In virtue of the commutation relation \eqref{commutes} we obtain that $|\cV^*|M_z =M_z|\cV^*|$, so  $|\cV ^*|=M_{\psi}$ for some $\psi\in L^\infty(\nu)$. 

We proved that $\ker |\cV^*|=\ker \cV^* =\{0\}$, so $\psi\ne 0 $ $\nu$-a.e.  
Let $1/h=\psi\in L^\infty(\nu)$. Since $\widetilde U^* = M_{h}\cV $, we need to show that $\widetilde U$ is a unitary operator.

We proved that $\ker \cV^* =\{0\}$. 
We also know that $\wt U$ restricted to $(\Ker \cV)^\perp$ is an isometry, so $\wt U$ is an isometry. 
We assumed that $\ker\cV=\{0\}$, and since $h\ne 0$ $\nu$-a.e.~we conclude that $\ker \wt U^*=\{0\}$. Therefore $\wt U$ is a unitary operator.

The second part of the theorem is now almost trivial. The fact that $\wt U^*=M_h \cV:L^2(\mu)\to L^2(\nu)$ is unitary means that $\cV$ is a unitary operator $L^2(\mu) \to L^2(|h|^2\nu)$. 
Together with the intertwining relationship \eqref{intertwining} and  the fact that  $\cV b =\1$ it means mean that $|h|^2\nu =\mu_\alpha$ and $\cV:L^2(\mu)\to L^2(\mu_\alpha)$ is the spectral representation of $U_\alpha$ with respect to the cyclic vector $b$, i.e.~that $\cV=\cV_\alpha$.
\end{proof}

\begin{rem*}
The assumption $\ker \cV=\{0\}$ is essential in the above rigidity theorem. To see that 
consider a measure $\mu$ which is a finite linear combination of atoms, 
$\mu= \sum_{k=1}^n w_k \delta_{\xi_k}$, and let $U_1$ be the multiplication by $\xi$ in $L^2(\mu)$. Let $U_\alpha= U + (\alpha-1)b b_1^*$, $|\alpha|=1$ be the rank $1$ perturbation considered in the beginning of Section \ref{s:V_alpha}, and let $\mu_\alpha$ be the spectral measure of $U_\alpha$ with respect to the cyclic vector $b$. Then by Theorem \ref{t-repr-V} the operator $\cV_\alpha:L^2(\mu)\to L^2(\mu_\alpha)$ satisfying \eqref{ComRel-02} is given by \eqref{repr-V}. 

Note that $\mu_\alpha$ is also a linear combination of atoms, $\mu_\alpha= \sum_{k=1}^n v_k \delta_{z_k}$. Note also that the points $z_k$ are zeroes of the function $1-\theta_0(z)$, and if we perturb even one of these point the formula \eqref{repr-V} would not give us a bounded operator. 

However, if we just remove one (or more) of the points $z_k$ to get a measure $\nu$, then \eqref{repr-V} still defines a bounded operator $L^2(\mu)\to L^2(\nu)$. It is also clear from comparing the dimensions that this operator has a non-trivial kernel. Therefore, it cannot be made into a unitary operator by composing it with any operator.

It is also clear that $M_z$ in $L^2(\nu)$ cannot be unitarily equivalent to the unitary operator $U_\alpha$ because these operators have different ranks. 
\end{rem*}

\providecommand{\bysame}{\leavevmode\hbox to3em{\hrulefill}\thinspace}
\providecommand{\MR}{\relax\ifhmode\unskip\space\fi MR }
\providecommand{\MRhref}[2]{%
  \href{http://www.ams.org/mathscinet-getitem?mr=#1}{#2}
}
\providecommand{\href}[2]{#2}

\end{document}